\documentclass[11pt]{amsart}
\usepackage{geometry}                
\usepackage{graphicx}
\usepackage{amssymb}
\usepackage{epstopdf}
\usepackage{amssymb, amsmath, amscd, amsthm,
color, epsfig
}
\usepackage[all]{xy}          
\xyoption{dvips}              

\title[Rational homology of spaces of embeddings]{On the rational homology of high dimensional analogues of spaces of long knots}
\author{Gregory Arone}
\address{Department of Mathematics \\ University of Virginia \\ Charlottesville, VA 22904 \\ USA}
\email{zga2m@virginia.edu}
\thanks{Both authors gratefully acknowledge NSF support via collaborative grant
DMS 0967649 and via Midwest Topology Network grant DMS 0844249}
\thanks{The first author was supported by the Fulbright fellowship}
\author{Victor Turchin}
\address{Department of Mathematics\\ Kansas State University \\ Manhattan, KS 66506 \\ USA}
\email{turchin@ksu.edu}
\date{}                                           

\newcommand{\calA}{\mathcal{A}}

\newcommand{\calC}{\mathcal{C}}
\newcommand{\calD}{\mathcal{D}}

\newcommand{\calO}{\mathcal{O}}
\newcommand{\calF}{\mathcal{F}}
\newcommand{\calP}{\mathcal{P}}

\newcommand{\balls}{\mathtt{B}}

\newcommand{\R}{{\mathbb R}}
\newcommand{\Z}{{\mathbb Z}}
\newcommand{\Q}{{\mathbb Q}}

\newcommand{\Koszul}{\operatorname{K}}

\newcommand{\holim}{\operatorname{holim}\,}
\newcommand{\hocolim}{\operatorname{hocolim}\,}

\newcommand{\sEmb}{\operatorname{sEmb}}
\newcommand{\Emb}{\operatorname{Emb}}
\newcommand{\Ebar}{\overline{\Emb}}
\newcommand{\Ebarc}{\Ebar_{\mathrm{c}}}

\newcommand{\Imm}{\operatorname{Imm}}
\newcommand{\Top}{\operatorname{Top}}

\newcommand{\one}{{\mathtt 1}}
\newcommand{\Com}{{\mathrm{Com}}}
\newcommand{\Epi}{\Omega}
\newcommand{\sur}{\operatorname{Sur}}
\newcommand{\map}{\operatorname{Map}}
\newcommand{\ET}{\operatorname{T}}

\newcommand{\hRmod}{\operatorname{hRmod}}
\newcommand{\hWBimod}{\operatorname{hInfBim}}
\newcommand{\chains}{\mathtt{C}}
\newcommand{\ce}{\operatorname{cr}}
\newcommand{\Ch}{\mathrm{Ch}}
\newcommand{\inj}{\operatorname{Inj}}
\newcommand{\Config}[2]{\mathrm{C}(#1, #2)}
\newcommand{\tildeHH}{\widetilde{\HH}}
\newcommand{\HH}{\operatorname{H}}

\newcommand{\ind}{\operatorname{ind}}
\newcommand{\st}{{\mathrm{st}}}
\newcommand{\const}{{\mathrm{c}}}
\newcommand{\mfld}{{\mathcal{M}}}
\newcommand{\smooth}{\mathrm{sm}}
\newcommand{\framed}{\mathrm{fr}}

\usepackage{graphicx, psfrag,rotating}
\newcommand{\Po}{\operatorname{Po}}

\newcommand{\Ebarmn}{{\overline{\mathrm{Emb}}}_c(\R^m,\R^n)}

\numberwithin{equation}{section}

\theoremstyle{plain}
\newtheorem{theorem}{Theorem}[section]
\newtheorem{proposition}[theorem]{Proposition}
\newtheorem{lemma}[theorem]{Lemma}
\newtheorem{corollary}[theorem]{Corollary}

\theoremstyle{definition}
\newtheorem{definition}[theorem]{Definition}
\newtheorem{example}[theorem]{Example}
\newtheorem{examples}[theorem]{Examples}

\theoremstyle{remark}
\newtheorem{remark}[theorem]{Remark}
\newtheorem*{remark*}{Remark}
\newtheorem*{remarks*}{Remarks}

\begin{document}

\begin{abstract}
We study high-dimensional analogues of spaces of long knots. These are spaces of compactly-supported embeddings (modulo immersions) of $\R^m$ into $\R^n$.  We view the space of embeddings as the value of a certain functor at $\R^m$, and we apply manifold calculus to this functor. Our first result says that the Taylor tower of this functor can be expressed as the space of maps between infinitesimal bimodules over the little disks operad. We then show that the formality of the little disks operad has implications for the homological behavior of the Taylor tower. Our second result says that when $2m+1<n$, the singular chain complex of these spaces of embeddings is rationally equivalent to a direct sum of certain finite chain complexes, which we describe rather explicitly. 
\end{abstract}

\maketitle

\tableofcontents

\section*{Introduction}
Let $\R^n$ be a Euclidean space and $\R^m\subset \R^n$ a subspace. Let $\Emb_\const(\R^m, \R^n)$ be the space of smooth embeddings of $\R^m$ into $\R^n$ that agree with the inclusion outside a bounded set. When $m=1$, this space is sometimes called the space of long knots in $\R^n$.
%

Similarly, let $\Imm_\const(\R^m, \R^n)$ be the space of smooth immersions of $\R^m$ into $\R^n$ that agree with the inclusion outside a bounded set. It follows from the Smale-Hirsh theory \cite{Hirsch} that there is a homotopy equivalence (assuming that $n\ge m+1$) $\Imm_\const(\R^m, \R^n)\simeq \Omega^m \inj(\R^m, \R^n)$ where $\Omega^m$ denotes $m$-fold loop space, and $\inj(\R^m, \R^n)$ is the Stiefel manifold of linear isometric injections of $\R^m$ into $\R^n$. Thus, the homotopy type of $\Imm_\const(\R^m, \R^n)$ is (in some sense) well-understood, and we view it as ``the easy part of $\Emb_\const(\R^m, \R^n)$''.

Let $\Ebarc(\R^m, \R^n)$ be the homotopy fiber of the inclusion map $\Emb_\const(\R^m, \R^n)\longrightarrow \Imm_\const(\R^m, \R^n)$. This space is the subject of this paper.
We view the space $\Ebarc(\R^m, \R^n)$ as a value of a contravariant functor
$$U\mapsto \Ebarc(U, \R^n),$$
where $U$ ranges over a certain category of open subsets of $\R^m$. Our first goal is to describe the Taylor tower of this functor in terms of operads. 
Our first theorem says that the Taylor tower can be described as a mapping space between {\it infinitesimal bimodules} over the little disks operad $\balls_m$. The definition of infinitesimal bimodules is reviewed in Section~\ref{section: operads modules}. Every operad is an infinitesimal bimodule over itself. The inclusion $\R^m\hookrightarrow \R^n$ induces a map of operads $\balls_m\longrightarrow \balls_n$, and in particular it endows $\balls_n$ with the structure of an infinitesimal bimodule over $\balls_m$. The following is our first main result (it is contained in Theorem~\ref{theorem: operadic}):
\begin{theorem}\label{theorem: first main}
There is an equivalence
$$\ET_\infty\Ebar_\const(\R^m, \R^n)\simeq \underset{\balls_m}{\hWBimod}(\balls_m, \balls_n).$$
\end{theorem}
Here $\hWBimod(-, -)$ denotes the derived mapping space between two infinitesimal bimodules over an operad. $\ET_\infty \Ebar_\const(\R^m, \R^n)$ is the limit of the Taylor tower of the functor $U\mapsto \Ebar_\const(U, \R^n)$. Taylor tower is taken in the sense of manifold calculus, or embedding calculus. More precisely, we need a variation of manifold calculus designed for functors with bounded support. It is similar to calculus of manifolds with boundary, developed in~\cite[Section 10]{WeissEmb} and~\cite[Section 9]{DeBrito-Weiss}. 

A deep convergence result of Goodwillie, Klein and Weiss implies that the natural map $$\Ebar_\const(\R^m, \R^n)\longrightarrow \ET_\infty \Ebar_\const(\R^m, \R^n)$$ is a weak homotopy equivalence if $m+3 \le n$. Combining this with Theorem~\ref{theorem: first main}, we conclude that the following natural map is an equivalence when $m+3\le n$
$$\Ebar_\const(\R^m, \R^n)\longrightarrow  \underset{\balls_m}{\hWBimod}(\balls_m, \balls_n).$$

The case $m=1$ of Theorem~\ref{theorem: first main} is related to Sinha's cosimplicial model for the space of long knots~\cite{Sinha}. Indeed, the homotopy totalization of Sinha's cosimplicial space can be interpreted as the mapping space 
$$\underset{\operatorname{Ass}}{\hWBimod}(\operatorname{Ass}, \widetilde\balls_n)$$
where $\operatorname{Ass}$ is the associative operad, which is equivalent to $\balls_1$, and $\widetilde\balls_n$ is an operad equivalent to $\balls_n$ that admits a map from the associative operad.

We also prove (or reprove) an analogous, but simpler statement about the homotopy fiber of the inclusion $\Emb(M, \R^n)\longrightarrow \Imm(M, \R^n)$, where $\Emb(M, \R^n)$ and $\Imm(M, \R^n)$ are spaces of all smooth embeddings and immersions respectively. We denote this homotopy fiber by $\Ebar(M, \R^n)$. We describe the Taylor tower of $\Ebar(M, \R^n)$ in the case when $M$ is an open subset of $\R^m$.  We show that the Taylor tower is equivalent to the space of maps between {\it right} modules over the operad $\balls_m$ (Theorem~\ref{theorem: operadic}).

\begin{remark} \label{rem: difference} In the first draft of the paper we proved a different version of Theorems~\ref{theorem: first main} and~\ref{theorem: operadic}. That version was based on the discretization of the little disks operad. The current form of these theorems was stated in the original version as a conjecture. It was suggested by one of the referees that we prove the theorems in their current form. 

In the meantime, a proof of the theorems in their current form was given by the second author in~\cite{Turchin13} by a different approach. In particular, that paper uses the Fulton-McPherson operad instead of the little disks operad. 

Also in the meantime, a result similar to the part of Theorem~\ref{theorem: operadic} that deals with right modules was proved by de Brito and Weiss in~\cite{DeBrito-Weiss}. Similar results were also obtained by D. Pryor in his Ph.D. thesis~\cite{Pryor}. The paper of de Brito - Weiss influenced the current version of this paper. Our proof strategy is similar to theirs. But we note that our statement and theirs differ somewhat. On one hand, de Brito and Weiss prove a more general statement than ours about the Taylor tower of $\Emb(M, \R^n)$, in that it applies to all manifolds $M$, and not just open subsets of $\R^m$. The ``price'' for this is that their result is stated in terms of modules over the {\it framed} little disks operads, while we work with the unframed version. For our purposes it is important to use the unframed disks operad.

Readers who are interested in the discretized version of the theorems can find them in the first version of this paper on Arxiv.
\end{remark}

There is a homological analogue of Theorem~\ref{theorem: first main}. Let $\chains(X)$ denote the normalized singular chain complex of $X$. We view $\chains(\Ebar_\const(\R^m, \R^n))$ as a value of a functor $U\mapsto \chains(\Ebar_\const(U, \R^n))$. The homological analogue of Theorem~\ref{theorem: first main} says that there is an equivalence
\begin{equation}\label{eq: homological}
\ET_\infty \chains\left(\Ebar_\const(\R^m, \R^n)\right) \simeq  \underset{\chains(\balls_m)}{\hWBimod}(\chains(\balls_m), \chains(\balls_n)).
\end{equation}
The proof of this assertion is essentially the same as of Theorem~\ref{theorem: first main}. However, the homological functor has considerably weaker convergence properties. The natural map
$$\chains\left(\Ebar_\const(\R^m, \R^n)\right) \longrightarrow \ET_\infty \chains\left(\Ebar_\const(\R^m, \R^n)\right) $$
is only known to be an equivalence when $2m+1<n$ \cite{WeissHomol}.

Our next step is to use the formality of the little disks operad (\cite{Kontsevich, LV}) to simplify the description of the homological Taylor tower over the reals (and therefore also over the rationals). This is similar in spirit to what was done in~\cite{LTV} for spaces of long knots and in~\cite{ALV} for the spaces $\Ebar(M, \R^n)$. 
 The formality theorem implies that if one works over $\R$, then the infinitesimal bimodule $\chains(\balls_n)$ in~\eqref{eq: homological} can be replaced with $\HH(\balls_n; \R)$. Here the action of $\chains(\balls_m)$ on $\HH(\balls_n)$ factors through the commutative operad. It follows that one can present the right hand side of~\eqref{eq: homological} as a mapping object between infinitesimal bimodules over the commutative operad. It turns out that infinitesimal bimodules over the commutative operad are the same thing as right modules over the category $\Gamma$ of of pointed finite sets (here by ``right modules over'' we mean ``contravariant functors whose domain is'') . Let $\Epi$ be the category of unpointed sets and surjective functions between them. A well-known result of Pirashvili says that there is an equivalence of categories between right $\Gamma$-modules and right $\Epi$-modules~\cite{PirashviliDold}. It follows that the Taylor tower of $\chains\left(\Ebar_\const(\R^m, \R^n)\right)\otimes \R$ can be described as a mapping complex between right $\Epi$-modules. By a counting argument, the same conclusion holds over $\Q$ as over $\R$. In a further simplification, we show that $\Ch$ can be replaced with $\HH$ in the source variable as well as the target. The following is our second main result (Theorem~\ref{theorem: main theorem})
\begin{theorem}\label{theorem: second main}
Assume that $n> 2m+1$. There is a weak equivalence of chain complexes
\begin{equation}\label{eq: homology decomposition}
\chains^\Q(\Ebarc(\R^m, \R^n)) \simeq\underset{\Epi}{\hRmod}\left(\tildeHH(S^{m-}), \hat\HH(\balls_n;\Q)\right).
\end{equation}
\end{theorem}
Here $\chains^\Q(X)=\chains(X)\otimes \Q$, $\tildeHH(S^{m-})$ is the right $\Epi$-module defined by the fomula $n\mapsto \tildeHH(S^{mn})$, and $\hat\HH(\balls_n;\Q)$ is the cross-effect of $\HH(\balls_n; \Q)$ as defined by Pirashvili. $\hRmod$ stands for the derived mapping complex of right modules.

\begin{remark*} If one works integrally, then the equivalence~\eqref{eq: homology decomposition} does not hold, but one still can show that there is a spectral sequence starting with the homology of the right hand side and converging to the homology of the left hand side. We hope to come back to this in another paper.
\end{remark*}

We view Theorem~\ref{theorem: second main} as giving a rather explicit algebraic model for the rational homology of $\Ebarmn$. Let $\HH_j(X)$ be the $j$-th homology group of $X$, considered as a chain complex concentrated in dimension $j$. There are obvious isomorphisms of $\Epi$-modules
$$\hat\HH(\balls_n)\cong \bigoplus_{t=0}^\infty \hat\HH_{(n-1)t}(\balls_n)\cong \prod_{t=0}^\infty \hat\HH_{(n-1)t}(\balls_n)$$
and
$$\widetilde\HH(S^{m-})\cong \bigoplus_{s=0}^\infty \widetilde\HH_{ms}(S^{m-})\cong \prod_{s=0}^\infty  \widetilde\HH_{ms}(S^{m-}).$$
Here we have used the well-known fact that $\hat\HH_j(\balls_n(k))$ is non-zero only if $j$ is a multiple of $n-1$ and the obvious fact that $\widetilde \HH_i(S^{m-})$ is non-zero only if $i$ is a multiple of $m$. It follows that the right hand side of~\eqref{eq: homology decomposition} splits as a product of mapping spaces between right $\Epi$-modules with values in chain complexes concentrated in a single degree. More precisely, there is a weak equivalence of chain complexes
\begin{equation}\label{eq: basic decomposition for homology}
\underset{\Epi}{\hRmod}\left(\tildeHH(S^{m-}), \hat\HH(\balls_n;\Q)\right)\simeq \prod_{s,t\ge 0} \underset{\Epi}{\hRmod}\left(\tildeHH_{ms}(S^{m-}), \hat\HH_{(n-1)t}(\balls_n;\Q)\right).\end{equation}

In fact, one can show that if $n>2m+1$ then the direct product on the right side can be replaced with direct sum (Remark~\ref{remark: relevant degrees}). 
Also, the decomposition can be reduced even further. For example, we show that the terms on the right hand side of~\eqref{eq: basic decomposition for homology} are
 non-zero only if $s\le 2t$.

Isomorphism \eqref{eq: basic decomposition for homology} produces a double splitting in the rational homology of $\Ebarc(\R^m, \R^n)$. In the case $m=1$, this double splitting is equivalent to the one studied in~\cite{Turchin10}.

In Sections~\ref{section: Koszul spectral sequence}-\ref{section: Koszul complex} we  analyze the mapping complexes
$$
\underset{\Epi}{\hRmod}\left(\tildeHH_{ms}(S^{m-}), \hat\HH_{(n-1)t}(\balls_n;\Q)\right)$$ by filtering the category $\Epi$ by cardinality. Analysing this filtration leads us to a certain explicit finite complex that is weakly equivalent to this mapping complex. We call it the Koszul complex, because it can be viewed as a kind of Koszul resolution.

In Section~\ref{section: forests} we give an explicit description of the Koszul complex in terms of certain spaces of forests.  It is quite interesting, although not really surprising, that the graph-complexes that we obtain at this point look similar to those obtained from the Bott-Taubes generalized construction used to study the De Rham cohomology of higher dimensional knot spaces~\cite{Catt,Sakai}.

It also is worth noting that the right $\Epi$-modules $\tildeHH(S^{m-})$ and  $\hat\HH(\balls_n;\Q)$ are, up to shifts of degree, almost independent of $m$ and $n$. More precisely, they only depend on the parity of $m$ and $n$. It follows that the total group $\HH(\Ebarc(\R^m, \R^n);\Q)$ depends, in some sense, only on the parities of $m$ and $n$. For all $m$ and $n$ of a fixed parity (and satisfying $n>2m+1$), the total group is built of the same ingredients. But the topological dimensions of the different ingredients depend on $m$ and $n$.

We also prove a companion result to Theorem~\ref{theorem: second main} for spaces of the form $\Ebar(M, \R^n)$ (Proposition~\ref{prop: getting there}). We view it as a reformulation of the main result of~\cite{ALV}
\begin{theorem}\label{theorem: gamma}
Suppose  that $M$ is an open subset of $\R^m$ and $2m+1 < n$. There is an equivalence of chain complexes
$$\chains^\Q(\Ebar(M,\R^n))\simeq \underset{\Com}{\hRmod}\left(\chains^\Q(M^-),\HH\left(\balls_n;\Q\right)\right).$$
\end{theorem}
Here $\chains^\Q(M^-)$ is the right $\Com$-module determined by the formula $k\mapsto \chains^\Q(M^k)$. The theorem implies that in high enough codimension, $\HH(\Ebar(M,\R^n); \Q)$ is a homotopy functor of $M$. In particular, Theorem~\ref{theorem: gamma} has the following consequence, which also was proved in~\cite{ALV}.
\begin{corollary}
Let $M_1$, $M_2$ smooth manifolds that embed into $\R^m$. Suppose that $2m+1<n$. Also suppose that $M_1$ and $M_2$ are related by a chain of maps inducing an isomorphism in rational homology. Then there is an isomorphism
$$\HH(\Ebar(M_1,\R^n); \Q)\cong\HH(\Ebar(M_2, \R^n); \Q).$$
\end{corollary}
\begin{proof}
Replacing $M_1$ and $M_2$ with tubular neighborhoods, we may assume that they are open subsets of $\R^m$. We can do it because inclusion into a tubular neighborhood induces an equivalence on $\Ebar(-, \R^n)$. If $M_1$ and $M_2$ are related by a chain of real homology equivalences, then the right $\Com$-modules $\chains^\Q(M_1^-)$ and $\chains^\Q(M_2^-)$ are weakly equivalent (in the category of right $\Com$-modules with values in rational chain complexes). By the theorem, the homotopy type of the right $\Com$-module $\chains^\Q(M^-)$ determines $\HH(\Ebar(M,\R^n); \Q)$.
\end{proof}
\subsubsection*{A section by section outline of the paper} In Section~\ref{section: complexes prelims} we review some preliminaries about chain complexes. The main goal of this section is to prove that a certain explicit model for Postnikov sections in the category of chain complexes is an enriched functor. This is used later in Section~\ref{section: applying formality}. In Sections~\ref{section: operads modules} and~\ref{section: operads as categories} we recall some basics about operads and their modules. In particular, we review the concept of an infinitesimal bimodule over an operad, which is perhaps not very well-known. We show that the category of right modules over an operad, as well as the category of infinitesimal bimodules, are equivalent to certain categories of diagrams. In Sections~\ref{section: little disks} we apply the general theory to obtain a convenient description of the categories of right modules and of infinitesimal bimodules over the little disks operad. In Section~\ref{section: module maps} we show how the Taylor towers for  $\chains(\Ebar(M, \R^n))$ and $\chains(\Ebarmn)$ can be described in terms of maps between, respectively, right modules and infinitesimal bimodules over the little disks operad. In Section~\ref{section: applying formality} we start working over the reals/rationals, and we show how Kontsevich's theorem on the formality of the little disks operad can be used to drastically simplify our models for $\chains^\Q(\Emb(M, \R^n))$ and $\chains^\Q(\Ebarmn)$. In Section~\ref{section: change of operads} we show that the model for $\chains^\Q(\Ebarmn)$ can be simplified even further, and rewritten in terms of maps between right $\Epi$-modules. This enables us to prove Theorem~\ref{theorem: main theorem}.

In the easy Section~\ref{section: splitting} we derive the equivalence~\eqref{eq: basic decomposition for homology}. In Section~\ref{section: Koszul spectral sequence} we show how filtering the category $\Epi$ by cardinality gives rise to a spectral sequence (which we call the Koszul spectral sequence) for calculating the homology groups of the complex of maps between right $\Epi$-modules. After some preparatory work in Sections~\ref{section: Koszul dual} and~\ref{section: configurations spaces revisited} we show,  in Section~\ref{section: Koszul complex}, that in the cases of interest to us the first term of the Koszul spectral sequence consists of a single chain complex. We call these chain complexes Koszul complexes. Thus when $2m+2\le n$ the chain complex $\chains^\Q(\Ebarmn)$ is equivalent to a direct sum of Koszul complexes. We describe the Koszul complexes explicitly in Section~\ref{section: forests} as complexes generated by certain types of forests.
\subsubsection*{Acknowledgements} We would like to warmly and sincerely thank the referees,  who undoubtedly invested much time and effort in reading the paper. They made a number of valuable suggestions, spread over several reports, that led to significant improvements (see, for example, Remark~\ref{rem: difference} above). 


The second author is grateful to the University of Virgina for its hospitality
during several visits that he made while working on this project.

\section{Chain complexes and Postnikov sections}\label{section: complexes prelims}
In this paper, $\Ch$ denotes the category of non-negatively graded chain complexes. Sometimes we need to consider the category of complexes of modules over a commutative ring with unit. When we want to emphasize the ground ring, we write $\Ch^R$ for the category of chain complexes of $R$-modules. 

It is well-known that $\Ch$ is a closed symmetric monoidal category, and it also has a Quillen model structure compatible with the monoidal structure (the projective model structure). If $A$ and $B$ are objects of $\Ch$, then by $\hom(A, B)$ we denote the internal mapping object from $A$ to $B$. Thus $\hom(A, B)$ is itself an object of $\Ch$. In degree zero, $\hom(A, B)$ is the group of chain homomorphisms from $A$ to $B$, and $\HH_0(\hom(A, B))$ is the group of chain homotopy classes of chain homomorphisms.

Let $\chains\colon \Top\longrightarrow \Ch$ be the normalized singular chains  functor. 
It is well-known that $\chains$ is a lax symmetric monoidal functor (the original reference is probably~\cite[Theorem 5.4]{EM}). The functor $\chains$ defines a tensoring and cotensoring of the category $\chains$ over $\Top$. Thus if $X$ is a space and $A$ is a chain complex, we define $A\otimes X:=A\otimes \chains(X)$ and $\map(X, A)=\hom(\chains(X), A)$. 

\subsubsection*{Postnikov sections} We will use a small model for Postnikov sections in $\Ch$, that was also used in \cite{ALV}. 
\begin{definition}
Let $P$ be a chain complex. Define $\Po_n(P)$ to be the following chain complex
$$\Po_n(P)_i=\left\{
\begin{array}{cc} 
P_i & i\le n \\
d(P_{n+1}) & i = n+1 \\
0 & i> n+1
\end{array}\right.$$
Here $d(P_{n+1})$ denotes the $n$-dimensional boundaries in $P$. The differential in $\Po_n(P)$ is the same as in $P$ in dimensions $\le n$ and is given by the inclusion $d(P_{n+1})\hookrightarrow P_{n}$ in dimension $n+1$.
\end{definition}
There is an obvious chain homomorphism $P \longrightarrow \Po_n(P)$, natural in $P$. It induces an isomorphism on $\HH_i$ for $i\le n$ (and it is a chain level isomorphism in dimensions $i<n$). On the other hand $\HH_i(\Po_n(P))=0$ for $i>n$. Thus $\Po_n(P)$ is a model for the $n$-th Postnikov section of $P$. Clearly there are compatible chain homomorphism $\Po_n(P)\longrightarrow \Po_{n-1}(P)$, inducing isomorphism on homology in dimensions $<n$. These homomorphisms are fibrations in $\Ch$ (because they are surjective). Together, they give a model for the Postnikov tower of $P$. Let $k_n(P)=\ker(\Po_n(P) \longrightarrow \Po_{n-1}(P))$. $k_n(P)$ is a chain complex concentrated in degrees $n$ and $n+1$. The unique non-trivial boundary homomorphism in $k_n(P)$ is the inclusion of the $n$-dimensional boundaries of $P$ into the cycles. Finally, let $\HH_n(P)$ be the $n$-th homology of $P$, considered as a chain complex concentrated in dimension $n$. We have a zig-zag of chain homomorphisms, natural in $P$.
$$P\longrightarrow \Po_n(P) \longleftarrow k_n(P) \stackrel{\simeq}{\longrightarrow}\HH_n(P).$$
We need the following simple observation about these functors. It is equivalent to saying that all these functors are {\em enriched} over $\Ch$, and the natural transformations between them are enriched transformations. 
\begin{lemma}\label{lemma: postnikov enriched}
Let $X$, $P$ and $Q$ be non-negatively graded chain complexes. Suppose we have a chain homomorphism $X\otimes P \longrightarrow Q$. Then there are unique vertical homomorphisms (determined by the first one) that make the following diagram commute
$$\begin{array}{ccccccc}
X\otimes P& \longrightarrow &X\otimes \Po_n(P) &\longleftarrow &X\otimes k_n(P)& {\longrightarrow}& X\otimes\HH_n(P) \\
\downarrow & & \downarrow & & \downarrow & & \downarrow \\
Q& \longrightarrow &\Po_n(Q) &\longleftarrow &k_n(Q)& {\longrightarrow}& \HH_n(Q)
\end{array}$$
\end{lemma}
\begin{proof}
We begin by constructing a chain map $X\otimes \Po_n(P)\longrightarrow \Po_n(Q)$. The homomorphism $P\longrightarrow \Po_n(P)$ is surjective, and therefore the homomorphism $X\otimes P\longrightarrow X\otimes \Po_n(P)$ is surjective, by right exactness of tensor product. We need to show that the kernel of this homomorphism goes to zero in $\Po_n(Q)$ under the composed map $X\otimes P\longrightarrow Q\longrightarrow \Po_n(Q)$. The map $X\otimes P\longrightarrow X\otimes \Po_n(P)$ is a direct sum of maps of the form $X_i\otimes P_j\longrightarrow X_i\otimes \Po_n(P)_j$, so its kernel is a direct sum of kernels. For $j\le n$, the above map is an isomorphism, so the kernel is trivial. For $j>n+1$, $X_i\otimes P_j$ has dimension $> n+1$, so it goes to zero in $\Po_n(Q)$. It remains to consider summands of the form $X_i\otimes P_{n+1}$. Again, if $i>0$, the summand goes to zero in $\Po_n(Q)$ for dimensional reasons. It remains to consider the summand $X_0\otimes P_{n+1}$. We need to show that any element in the kernel of the homomorphism $X_0 \otimes P_{n+1} \longrightarrow X_0 \otimes d(P_{n+1})$ goes to zero in $\Po_n(Q)$. Any element of  $X_0 \otimes P_{n+1}$ can be written in the form $\Sigma x_0^\alpha \otimes p_{n+1}^\alpha$. The assumption that it is in the kernel means that $\Sigma x_0^\alpha \otimes d(p_{n+1}^\alpha)=0$. But since the elements $x_0^\alpha$ are in dimension zero, it follows that 
$$ d\left(\Sigma x_0^\alpha \otimes p_{n+1}^\alpha\right)= \Sigma x_0^\alpha \otimes d(p_{n+1}^\alpha)=0.$$
So every element of the kernel is an $n+1$-dimensional cycle in $X\otimes P$. It follows that it goes to an $n+1$-cycle in $Q$, and therefore it goes to zero in $\Po_n(Q)$.

Next, we need to show the existence of a unique map $X\otimes k_n(P) \longrightarrow k_n(Q)$ that makes the second square commute. Since the homomorphism $k_n(Q) \longrightarrow \Po_n(Q)$ is a monomorphism, we need to show that the image of the composed homomorphism
\[
X\otimes k_n(P) \longrightarrow X\otimes \Po_n(P) \longrightarrow \Po_n(Q)
\]
is contained in the image of $k_n(Q)$. Since $\Po_n(Q)$ is zero above dimension $n+1$, and $X\otimes k_n(P)$ is zero in dimensions below $n$, we only need to consider dimensions $n$ and $n+1$. The homomorphism $k_n(Q) \longrightarrow \Po_n(Q)$ is an isomorphism in dimension $n+1$, so it remains to check dimension $n$. Let $Z_n(P)$ be the group of $n$-dimensional cycles of $P$. The $n$-dimensional part of $X\otimes k_n(P)$ is $X_0 \otimes Z_n(P)$. We need to show that the image of $X_0\otimes Z_n(P)$ in $\Po_n(Q)$ is contained in the group of cycles $Z_n(Q)$. For this it is enough to show that its image in $X\otimes \Po_n(P)$ consists of cycles. Let $\Sigma x_0^\alpha\otimes p_n^\alpha$ be an element of $X_0\otimes Z_n(P)$. Then its differential in $X\otimes \Po_n(P)$ is $$\Sigma d(x_0^\alpha) \otimes p_n^\alpha + x_0\otimes d(p_n^\alpha)$$
which is clearly zero. So $\Sigma x_0^\alpha\otimes p_n^\alpha$ is a cycle in $X\otimes \Po_n(P)$.

Lastly, we need to show that there is a unique homomorphism $X\otimes \HH_n(P) \longrightarrow \HH_n(Q)$ that makes the rightmost square commute. By right exactness of tensor product, the homomorphism $X\otimes k_n(P) \longrightarrow X\otimes \HH_n(P)$ is surjective, so we need to show that the kernel of this homomorphism goes to zero in $\HH_n(Q)$. Since $\HH_n(Q)$ is concentrated in dimension $n$, we only need to check dimension $n$. The homomorphism from $k_n(P)$ to $\HH_n(P)$ is, in dimension $n$, the surjection from the $n$-cycles of $P$ to the $n$-th homology of $P$. Its kernel is the group of boundaries $d(P_{n+1})$. Using the right exactness of tensor product one more time, we conclude that any element in the kernel of $X\otimes k_n(P) \longrightarrow X\otimes \HH_n(P)$ in dimension $n$ can be written in the form $\Sigma x_0^\alpha \otimes d(p_{n+1}^\alpha)$. By the usual calculation, this is the same as $d\left(\Sigma x_0^\alpha \otimes p_{n+1}^\alpha\right)$, so it is a boundary in $X\otimes k_n(P)$. Therefore it goes to a boundary in $k_n(Q)$ and to zero in $\HH_n(Q)$.
\end{proof}
\begin{corollary}
Let $\calC$ be a category enriched over $\Ch$. Let $F\colon \calC \longrightarrow \Ch$ be an enriched functor. Then there is a chain of enriched natural transformations between enriched functors.
\[
F \longrightarrow \Po_n(F) \longleftarrow k_n(F) \longrightarrow \HH_n(F).
\]
\end{corollary}
\begin{proof}
Let $i, j$ be objects of $\calC$. Let $\calC(i, j)$ be the chain complex of morphisms from $i$ to $j$. To say that $F$ is an enriched functor means that for all $i, j$ there are chain homomorphisms
$$ \calC(i, j) \otimes F(i) \longrightarrow F(j).$$
The homomorphisms are associative and unital. By Lemma~\ref{lemma: postnikov enriched}, there are analogous maps where $F$ is replaced with $\Po_n(F)$, $k_n(F)$ and $\HH_n(F)$, that are compatible with the natural transformations between these functors. The uniqueness part of the lemma guarantees that the induced maps are associative and unital.
\end{proof}
\begin{corollary}\label{cor: formal}
With the same notation as in the previous corollary, suppose that $F$ takes values in chain complexes that are homologically concentrated in degree $n$. Then $F$ is related to $\HH_n(F)$ by a chain of enriched natural equivalences (i.e., quasi-isomorphisms).
\end{corollary}
\begin{proof}
It is clear that if $F(i)$ is a chain complex whose homology is concentrated in dimension $n$, then all the natural transformations
$$F(i)\longrightarrow \Po_n(F(i)) \longleftarrow k_n(F(i)) \longrightarrow \HH_n(F(i))$$
are quasi-isomorphisms.
\end{proof}
In the terminology of~\cite{ALV} we can say that an enriched functor that takes values in complexes of homological dimension $n$ is {\em formal} in a natural and enriched sense.

\section{Operads, modules and infinitesimal bimodules}\label{section: operads modules}
\subsubsection{Definition of operads} In this section we will review the (very well-known) notions of an operad and a right module over an operad. We also will review the concept of an infinitesimal bimodule over an operad, which is less well-known.

For a general introduction to the theory of operads, given from a modern perspective, we suggest the recent book of Loday and Vallette~\cite{LodayVallette}. 

\begin{definition}
Let $\Sigma$ be the category of finite sets and isomorphisms between them. Let $\calC$ be any category. A {\em symmetric sequence} $P$ in $\calC$ is a functor $P\colon\Sigma\to \calC$.
\end{definition}
We will write $P(n)$ for $P(\{1,\ldots,n\})$ (and $P(0)$ for $P(\emptyset)$). In practice, a symmetric sequence is determined by the sequence $P(0), P(1), \ldots$ of objects in $\calC$ together with an action of the symmetric group $\Sigma_n$ on $P(n)$ for each $n$.
 \begin{definition} Let $(\calC, \otimes, \one)$ be a symmetric monoidal category.
An {\em operad} $O$ in $\calC$ is a symmetric sequence $O$ in $\calC$ equipped with a {\em unit map} $\eta\colon\one\to O(1)$ and {\em partial composition maps}
$$-\circ_a-\colon O(A)\otimes O(B)\longrightarrow O(A\cup_aB)$$
for each pair of sets $A$ and $B$ and each $a\in A$ where $A\cup_a B:=(A\setminus\{a\})\coprod B$. The composition maps must be natural with respect to isomorphisms of $(A, a)$ and of $B$, and must satisfy certain axioms that say that
\begin{enumerate}
\item the composition is associative in the sense that for all sets $A, B$ and $C$, and elements $a\in A$, $b\in B$, the two compositions $(-\circ_a-)\circ_b -$ and $-\circ_a(-\circ_b-)$ define the same map
$$O(A)\otimes O(B) \otimes O(C) \longrightarrow O(A\cup_a B\cup _b C).$$ \label{OperadAssociative}
\item for two distinct elements $a, a'\in A$, the operations $-\circ_a-$ and $-\circ_{a'}-$ commute.\label{OperadCommutative}
\item $\eta$ acts as both a right unit and a left unit with respect to the composition maps.
\end{enumerate}
\end{definition}

The axioms are spelled out fully in~\cite[Definition 2.2]{Ching} (for example).  A little below, we will spell out the analogous axioms for a module over an operad. 
\begin{remark}\label{rem: composition}
Let $O$ be an operad, and let $\alpha\colon A\to B$ be a function between sets. The operad structure on $O$ gives rise to a map
\[
O(B)\otimes\bigotimes_{b\in B} O(\alpha^{-1}(b)) \longrightarrow O(A).
\]
The map is defined by applying the composition operation $\circ_b$ for each $b\in B$. The commutativity hypothesis in the definition of operad ensures that the map is independent of the order in which the composition operations are performed. 
\end{remark}

\subsubsection{Modules} Let $\calC$ be a symmetric monoidal category. Let $\calD$ be a category tensored over $\calC$. In practice, we will be most interested in the case when $\calD$ and $\calC$ are both the category of chain complexes. But it will be convenient to have a more general set-up.
\begin{definition} Let $O$ be an operad in $\calC$. A {\em right module} over $O$ with values in $\calD$ is a symmetric sequence $M$ in $\calD$ together with partial composition maps
$$-\circ_a -\colon M(A)\otimes O(B)\longrightarrow M(A\cup_a B)$$
where $A$ and $B$ are finite sets and $a\in A$. The composition maps are required to be natural in isomorphisms of $(A, a)$ and $B$, and satisfy the following axioms
\begin{enumerate}
\item For all finite sets $A, B$ and $C$ and for all $a\in A$ and $b\in B$ the following diagram commutes
$$\begin{CD}
M(A)\otimes O(B) \otimes O(C) @>1\otimes(-\circ_b-)>> M(A)\otimes O(B\cup_b C) \\
@V(-\circ_a-)\otimes 1VV  @VV-\circ_a-V \\
M(A\cup_a B)\otimes O(C) @>>-\circ_b -> M(A\cup_a B\cup_b C)
\end{CD}.$$
\item For all finite sets $A$, $B$ and $B'$ and all $a, a'\in A$ where $a\ne a'$, the following diagram commutes
$$\begin{CD}
M(A)\otimes O(B) \otimes O(B')  @>(-\otimes_{a'}- )\otimes 1 >> M(A\cup_{a'} B')\otimes O(B) \\
@V(-\otimes_{a}- )\otimes 1VV  @VV-\otimes_{a}-V \\
M(A\cup_a B)\otimes O(B') @>> -\otimes_{a'} -> M(A\cup_a B\cup_{a'} B')
\end{CD}$$
(more precisely, the top arrow in this square needs to be preceded by the map switching the second and third factors).
\item For all finite sets $A$ and for all $a\in A$ the following morphism is the identity
$$M(A)=M(A)\otimes \one \stackrel{1\otimes \eta}{\longrightarrow} M(A)\otimes O(1) \stackrel{-\circ_a-}{\longrightarrow} M(A).$$
Here we have used the identification of $\{1\}$ with $\{a\}$.
\end{enumerate}
\end{definition}
\begin{remark}\label{rem: right-module}
Let $M$ be a right module over $O$. Let $\alpha\colon A\to B$ be a function between sets. One obtains a map
\[
M(B)\otimes\bigotimes_{b\in B} O(\alpha^{-1}(b)) \longrightarrow M(A).
\]
The map is defined by applying the composition operation $\circ_b$ for each $b\in B$. The commutativity hypothesis in the definition of a right module ensures that the map is independent of the order in which the composition operations are performed. Compare with Remark~\ref{rem: composition}. 
\end{remark}

%
%

The definitions above are well known. Now we will introduce a few definitions that are not so standard. As before, let $O$ be an operad in a symmetric monoidal category $\calC$, and let $\calD$ be a category tensored over $\calC$.
\begin{definition}
An {\em infinitesimal left module} over $O$ with values in $\calD$ is a symmetric sequence $M$ in $\calD$ together with partial composition maps, defined for all finite sets $A, B$ and all $a\in A$
$$-\circ_a-\colon O(A)\otimes M(B)\longrightarrow M(A\cup_a B).$$
The composition maps are required to be natural in isomorphisms of $(A, a)$ and $B$, and satisfy the following axioms
\begin{enumerate}
\item For all finite sets $A, B$ and $C$ and for all $a\in A$ and $b\in B$ the following diagram commutes
$$\begin{CD}
O(A)\otimes O(B) \otimes M(C) @>1\otimes(-\circ_b-)>> O(A)\otimes M(B\cup_b C) \\
@V(-\circ_a-)\otimes 1VV  @VV-\circ_a-V \\
O(A\cup_a B)\otimes M(C) @>>-\circ_b -> M(A\cup_a B\cup_b C)
\end{CD}.$$
\item For all finite sets $A$ the following morphism is the identity
$$M(A)=\one \otimes M(A) \stackrel{\eta\otimes 1}{\longrightarrow} O(1)\otimes M(A) \stackrel{-\circ_1-}{\longrightarrow} M(A).$$
\end{enumerate}
\end{definition}

\begin{remark}
The concept of an infinitesimal left module is different from the usual concept of a left module over an operad that is more commonly found in literature. A left module is usually defined to be a symmetric sequence $M$, equipped with maps of the following kind. Let $\alpha\colon A\to B$ be a function of sets. Then for a left module there would be a morphism
$$O(B)\otimes \bigotimes_{b\in B}M(\alpha^{-1}(b))\longrightarrow M(A).$$
The structure of an infinitesimal left module does not give rise to such a map (contrast with Remark~\ref{rem: right-module}). Neither does a left module structure automatically give rise to an infinitesimal left module structure. \end{remark}

Finally, we are ready for the key definition of this section. Let $O$ be an operad in $\calC$ and let $\calD$ be a category tensored over $\calC$
\begin{definition}  \label{def: infinitesimal bimodule}
An {\em infinitesimal bimodule} over $O$ with values in $\calD$ is a symmetric sequence $M$ in $\calD$ endowed with the structure of a right module over $O$ and of an infinitesimal left module over $O$. The right and left composition maps are required to satisfy the following axioms.
\begin{enumerate}
\item For all finite sets $A, B$, and $C$ and for all $a\in A, b\in B$, the following diagram commutes.
$$\begin{CD}
O(A)\otimes M(B) \otimes O(C) @>(-\circ_a-)\otimes 1>> M(A\cup_a B)\otimes O(C) \\
@V 1\otimes (-\circ_b-)VV @VV-\circ_b - V \\
O(A)\otimes M(B\cup_b C) @>> -\cup_a - > M(A\cup_a B\cup_b C)
\end{CD}. $$
\item For all finite sets $A, B, B'$ and for all distinct elements $a, a'\in A$, the following diagram commutes
$$\begin{CD}
O(A) \otimes O(B) \otimes M(B') @> (-\circ_{a'}-)\otimes 1 >> M(A\cup_{a'} B')\otimes O(B) \\
@V(-\circ_a -)\otimes 1 VV @VV-\circ_a- V \\
O(A\cup_a B)\otimes M(B') @>>-\circ_{a'}- > M(A\cup_a B \cup_{a'} B') \end{CD}.
$$
More precisely, the top map in the above diagram needs to be preceded by switching the second and third factors.
\end{enumerate}
\end{definition}
\begin{example}
Let $f\colon O\to P$ be a morphism of operads. Then $f$ endows $P$ with the structure of an infinitesimal bimodule over O. 
In particular, every operad is an infinitesimal bimodule over itself.
\end{example}

\section{Operads as categories, modules as functors}\label{section: operads as categories}
It is well known that the category of right modules over an operad is equivalent to a certain category of diagrams. In this section we review this construction and introduce an analogous one for the category of infinitesimal bimodules (by contrast, the category of honest bimodules, or even of left modules over an operad is not equivalent to a category of diagrams).

The category $\calF(O)$ of the following definition is equivalent to the category $\mathsf{cat}(O)$ defined in~\cite[5.4.1]{LodayVallette}. It is sometimes called the PROP associated with $O$.
Let $\calF$ be the category of finite sets and functions between them. For two finite sets $A$ and $B$, $F(A, B)$ denotes the set of functions from $A$ to $B$.
\begin{definition}
Let $O$ be an operad in a closed symmetric monoidal category $\calC$. $\calF(O)$ is a category enriched over $\calC$. The objects of  $\calF(O)$ are finite sets. For two finite sets $A, B$,  $\map_{\calF(O)}(A,B)$ is the object of $\calC$, defined by the following formula
$$\map_{\calF(O)}(A,B)=\coprod_{\alpha\in F(A, B)}\bigotimes_{b\in B} O(\alpha^{-1}(b)).$$

To define composition in $\calF(O)$ we need to describe maps
$$\map_{\calF(O)}(B,C)\otimes \map_{\calF(O)}(A,B)\longrightarrow \map_{\calF(O)}(A,C).$$ 
Since we assume that $\calC$ is a closed symmetric monoidal category, $\otimes$ distributes over coproducts in $\calC$. Therefore, it is enough to define composition for each summand of $\map_{\calF(O)}(A,B)$ with every summand of $\map_{\calF(O)}(B,C)$. Let $\alpha\colon A\to B$ and $\beta\colon B\to C$ be functions. The composition law in $\calF(O)$ is determined by  morphisms of the following form
\[\bigotimes_{c\in C} O(\beta^{-1}(c))\otimes \bigotimes_{b\in B} O(\alpha^{-1}(b)) \longrightarrow  \bigotimes_{c\in C} O((\beta\alpha)^{-1}(c)).\]
The morphism is defined by applying the composition operation $\cup_b$, for each $b\in B$. It is the product of morphisms of type described in Remark~\ref{rem: composition}. The associativity hypothesis in the definition of operad guarantees that composition in $\calF(O)$ is associative.

To define the unit morphisms in $\calF(O)$, note that $\map_{\calF(O)}(A,A)$ has a direct summand isomorphic to $O(1)^{\otimes A}$, corresponding to the identity function on $A$. The identity morphism of $A$ in $\calF(O)$ is defined by the map $\one \stackrel{\cong}{\longrightarrow} \one^{\otimes A} \stackrel{\eta^{\otimes A}}{\longrightarrow} O(1)^{\otimes A}$. The unicity hypothesis for operads guarantees that these behave as identity morphisms.
\end{definition}
\begin{example}\label{example: Com finite sets} Let $\Com$ be the commutative operad. Unless noted otherwise, $\Com$ denotes the commutative operad in the category of topological spaces, but we will use the same name for the commutative operad in any category. The value of the operad $\Com$ at every set $A$ is the unit object. All its structure maps are the identity morphism on the unit object. In the category of topological spaces, the unit object is also the final object $*$, so $\Com$ is the final object in the category of operads in $\Top$.
It is easy to see that $\calF(\Com)$ is the category of finite sets and functions between them. We will denote this category simply by $\calF$.
\end{example}
The point for us of the category $\calF(O)$ is that right $O$-modules are the same thing as contravariant functors from $\calF(O)$. The following lemma is elementary and well-known.
\begin{lemma}\label{lemma: right module}
Let $O$ be an operad in a closed symmetric monoidal category $\calC$. Let $\calD$ be a category enriched, tensored and cotensored over $\calC$. Then the category of right modules over $O$ with values in $\calD$ is equivalent to the category of enriched contravariant functors from $\calF(O)$ to $\calD$.
\end{lemma}
\begin{proof}[Sketch of proof]
Let $M$ be a right module over $\calO$. By remark~\ref{rem: right-module}, given a function $\alpha\colon A\to B$, there is a map
\begin{equation}\label{eq: summand}
M(B)\otimes \bigotimes_{b\in B} O(\alpha^{-1}(b)) \longrightarrow M(A).
\end{equation}
Taking sum over all functions from $A$ to $B$, one obtains a map
\begin{equation}\label{eq: functor}
M(B)\otimes \map_{\calF(O)}(A,B) \longrightarrow M(A).
\end{equation}
This makes $M$ into a contravariant functor from $\calF(O)$. The associativity hypothesis in the definition of a right module is equivalent to compatibility of these maps with composition in $\calF(O)$.

Conversely, suppose $M$ is an enriched contravariant functor from $\calF(O)$ to $\calD$. This means that there are maps as in~\eqref{eq: functor}. In particular, for a map $\alpha\colon A\to B$ we get a map as in~\eqref{eq: summand}. Now let $a\in A$ and let $\alpha\colon A\cup_a B\to A$ be the map that sends $B$ to $\{a\}$. We get a map
\[
M(A)\otimes O(1)^{\otimes A\setminus\{a\}} \otimes O(B)\longrightarrow M(A\cup_a B).
\]
Composing with the identity $\mathsf{1}\cong \mathsf{1}^{\otimes A\setminus \{a\}} \to O(1)^{\otimes A\setminus\{a\}}$, we obtain a map
\[
M(A) \otimes O(B)\longrightarrow M(A\cup_a B).
\]
This map endows $M$ with the structure of a right module. It is easy to check that the assumption that $M$ is a functor implies that that the composition maps are associative, commutative and unital, as required.

\end{proof}
\begin{corollary}\label{cor: right com modules are F modules}
The category of right $\Com$-modules with values in $\calD$ is equivalent to the category of contravariant functors from $\calF$ (the category of finite sets) to $\calD$.
\end{corollary}

Now we will introduce a variation of the category $\calF(O)$ that will do for infinitesimal bimodules over $O$ what $\calF(O)$ did for right modules. To begin with, let $\Gamma$ be the category of {\em pointed} finite sets and pointed functions. We will adopt the convention that all the pointed sets have the same basepoint $*$. So if $S$ is an object of $\Gamma$ then the elements of $S$ are $*, s_1, s_2,$ etc. For pointed sets $S, T$, let $\Gamma(S, T)$ be the set of pointed functions from $S$ to $T$. We define the category $\Gamma(O)$ as a restriction of $\calF(O)$ in an obvious way.
\begin{definition}
Let $O$ be an operad in a closed symmetric monoidal category $\calC$. $\Gamma(O)$ is a category enriched over $\calC$. The objects of  $\Gamma(O)$ are pointed finite sets. For two pointed finite sets $S, T$,  $\map_{\Gamma(O)}(S, T)$ is the object of $\calC$, defined by the following formula
$$\map_{\Gamma(O)}(S, T)=\bigoplus_{\alpha\in\Gamma(S, T)}\bigotimes_{t\in T} O(\alpha^{-1}(t)).$$
Composition in $\Gamma(O)$ is defined in the same way as in $\calF(O)$.
\end{definition}
\begin{example}\label{example: Com pointed sets} Recall that $\Com$ is the commutative operad in $\Top$. It is easy to see that $\Gamma(\Com)$ is just $\Gamma$. Compare with Remark~\ref{example: Com finite sets}.
\end{example}
For our purposes, it is important to define a certain twisted version of $\Gamma$.
\begin{definition}\label{def: twisted}
Let $O$ be an operad in $\calC$. We define a category $\widetilde\Gamma(O)$. The objects of $\widetilde\Gamma(O)$ are the same as of $\Gamma(O)$ (i.e., pointed finite sets). Moreover, for every two objects $S, T$, $\map_{\widetilde \Gamma(O)}(S, T)=\map_{\Gamma(O)}(S, T)$. The identity morphisms are defined in the same way. But the composition law is different. We proceed to describe the composition in $\widetilde\Gamma(O)$. Let $\alpha\colon S\to T$ and $\beta\colon T\to U$ be pointed functions. We need to define a morphism
\[
\bigotimes_{u\in U} O(\beta^{-1}(u))\otimes \bigotimes_{t\in T} O(\alpha^{-1}(t))\longrightarrow \bigotimes_{u\in U} O((\beta\alpha)^{-1}(u)).
\]
Recall that the composition in $\Gamma(O)$ is defined by performing the operations $\circ_t$, for all $t\in T$. Without loss of generality, suppose that we first perform the operation $\circ_*$. Note that the operation $\circ_*$ can be written as a map 
$$\circ_*\colon O(\beta^{-1}(*))\otimes O(\alpha^{-1}(*))\longrightarrow O(\alpha^{-1}(*)\vee \beta^{-1}(*)).$$
In $\widetilde\Gamma(O)$, this map is replaced with the composed map 
$$O(\beta^{-1}(*))\otimes O(\alpha^{-1}(*))\longrightarrow O(\alpha^{-1}(*))\otimes O(\beta^{-1}(*))\stackrel{\circ_*}{\longrightarrow} O(\alpha^{-1}(*)\vee \beta^{-1}(*)).$$
Here the first map is switching the order, and the second map is the map $\circ_*$. Note that the roles of $\alpha^{-1}(*)$ and $\beta^{-1}(*)$ have been reversed: in $\Gamma(O)$ the output of $O(\alpha^{-1}(*))$ is fed into the $*$-input of $O(\beta^{-1}(*))$, while in $\widetilde\Gamma(O)$ it is the other way around. Here we have been helped by the convention that all pointed sets share a common basepoint $*$.
 
The remaining operations  $\circ_t$ for $t\in T\setminus\{*\}$ are performed in $\widetilde \Gamma(O)$ in the same way as in $\Gamma(O)$.

%

It is tedious, but straightforward to check that the composition in $\widetilde \Gamma(O)$ is associative and unital.
\end{definition}
\begin{example}\label{example: gamma(com)=gamma}
In the case when $O=\Com$, $O(A)=*$ for all $A$. In this case the maps 
$$O(\beta^{-1}(*))\otimes O(\alpha^{-1}(*)) \longrightarrow O(\beta^{-1}(*)\vee \alpha^{-1}(*)).$$ 
and
$$O(\beta^{-1}(*))\otimes O(\alpha^{-1}(*)) \longrightarrow O(\alpha^{-1}(*))\otimes O(\beta^{-1}(*)) \stackrel{\circ_*}{\longrightarrow} O(\beta^{-1}(*)\vee \alpha^{-1}(*))$$
are clearly the same map $*\to  *$. Therefore $\Gamma(\Com)=\widetilde\Gamma(\Com)=\Gamma$. But in general, the categories $\Gamma(O)$ and $\widetilde\Gamma(O)$ are not equivalent. In particular, they are different when $O$ is the unframed little disks operad, which is the example that is important to us (we will get to it in the next section).

We suspect that $\Gamma(O)$ and $\widetilde\Gamma(O)$ are equivalent whenever $O$ is a cyclic operad. 
\end{example}

The following proposition is the main result of this section. It is analogous to Lemma~\ref{lemma: right module}.
\begin{proposition}\label{prop: infinitesimal bimodule}
Let $O$ be an operad in a closed symmetric monoidal category $\calC$. Let $\calD$ be a category enriched, tensored and cotensored over $\calC$. Then the category of infinitesimal bimodules over $O$ with values in $\calD$ is equivalent to the category of enriched contravariant functors from $\widetilde\Gamma(O)$ to $\calD$.
\end{proposition}
\begin{proof}
Let $M$ be an infinitesimal bimodule over $\calO$. We will associate with $M$ an enriched contravariant functor $\widetilde M\colon\widetilde\Gamma(O)\longrightarrow\calD$. It is defined on objects of $\widetilde\Gamma(O)$ by the formula
$$\widetilde M(S):= M(S\setminus\{*\}).$$

To describe the action of $\widetilde M$ on morphisms, we need to construct morphisms in $\calD$
$$\widetilde M(T)\otimes \map_{\widetilde\Gamma(O)}(S, T) \longrightarrow \widetilde M(S)$$
that are associative and unital with respect to composition in $\widetilde\Gamma(O)$. Since we assumed that $\calD$
is enriched, tensored and cotensored over $\calC$, it follows that tensor product distributes over coproducts, and so our task is equivalent to constructing morphisms
\[
\widetilde M(T) \otimes \bigotimes_{t\in T} O(\alpha^{-1}(t)) \longrightarrow \widetilde M(S)
\]
where, as usual, $\alpha\colon S\to T$ is a pointed function. The map is defined as the composite
\[
M(T\setminus\{*\}) \otimes \bigotimes_{t\in T} O(\alpha^{-1}(t)) \longrightarrow O(\alpha^{-1}(*))\otimes M(T\setminus\{*\}) \otimes \bigotimes_{t\in T\setminus \{*\}} O(\alpha^{-1}(t)) \longrightarrow M(S\setminus\{*\}).
\]
Here the first map is just  changing the order. For the second one, we use the map $$\cup_*\colon O(\alpha^{-1}(*))\otimes M(T\setminus\{*\})\to O(\alpha^{-1}(*)\cup_* T\setminus\{*\})$$ which comes from the infinitesimal left module structure on $M$. We also use the right module structure on $M$ to mulply $M(T\setminus\{*\})$ with $\bigotimes_{T\setminus\{*\}} O(\alpha^{-1}(t))$ on the right. It is routine to check that the axioms for infinitesimal bimodule imply that $\widetilde M$ is a well-defined functor. 

Conversely, suppose that $\widetilde M\colon \widetilde\Gamma(O)\to \calD$ is an enriched functor. We need to associate with it an infinitesimal bimodule $M$. Objectwise, $M$ is defined by $M(A)=\widetilde M(A_+)$, where $A_+=A\coprod\{*\}$. Note that adjoining a basepoint defines a functor $\calF(O)\to \widetilde\Gamma(O)$. This makes $M$ a contravariant functor from $\calF(O)$ to $\calD$. By Lemma~\ref{lemma: right module}, this endows $M$ with the structure of a right module over $O$. It remains to define the infinitesimal left module structure. We can phrase it as follows. Let $S$ be a pointed set and $B$ unpointed. We need to define maps $\circ_*$
\[
O(S)\otimes M(B)\longrightarrow M(S\cup_* B).
\]
Let $\alpha\colon S\coprod B\to B_+$ be the pointed map that sends $S$ to the basepoint and is the identity on $B$. Since $\widetilde M$ is a functor on $\widetilde \Gamma(O)$, we get a map
\[
\widetilde M(B_+)\otimes O(S)\longrightarrow \widetilde M(S\coprod B).\]
But this is the same thing as a map
\[
O(S)\otimes M(B) \longrightarrow M(S\cup_* B).\]
And this map defines the infinitesimal left module structure on $M$. Once again, it is straightforward to verify that $M$ is an infinitesimal bimodule, and that we defined a bijective corresponednce between infinitesimal bimodules and contravariant functors on $\widetilde \Gamma(O)$.

\end{proof}
The following corollary follows from Proposition~\ref{prop: infinitesimal bimodule} together with Example~\ref{example: gamma(com)=gamma}
\begin{corollary}\label{cor: weak com bimods are gamma mods}
The category of infinitesimal $\Com$-bimodules with values in $\calD$ is equivalent to the category of contravariant functors from $\Gamma$ (the category of pointed finite sets) to $\calD$.
\end{corollary}
\begin{example}
Let $X$ be a pointed topological space. Such an $X$ gives rise to a contravariant functor from $\Gamma$ to $\Top$
$$S\mapsto \map_*(S, X).$$
By the above corollary, this contravariant functor gives rise to an infinitesimal $\Com$-bimodule. Indeed, as a symmetric sequence it is given by the formula
$$A\mapsto X^{A}$$
(here $A$ lives in unpointed sets). 
The right module structure is given by the contravariant fuctoriality in the variable $A$. The infinitesimal left module structure is given by basepoint inclusion.
\end{example}

\section{The little disks operad} \label{section: little disks}
In this section we apply the theory of previous section to the little disks operad. Inevitably, we have to begin with a few definitions.

\begin{definition}\label{def: standard embeddings}
Let $\R^m$ be a Euclidean space. A {\em standard isomorphism} of $\R^m$ is a self homeomorphism of $\R^m$ that is the composition of a translation and a multiplication by a positive scalar.

Let $A$ be a connected subspace of $\R^m$. A map $f\colon A\to \R^m$ is called a {\em standard embedding} if $f$ equals the inclusion followed by a standard isomorphism of $\R^m$.

More generally, if $X$ is another subset of $\R^m$, then a standard embedding of $A$ into $X$ is a standard embedding of $A$ into $\R^m$ whose image lies in $X$.

Even more generally, we define the following category $\mfld$ of spaces and standard embeddings between them. An object of $\mfld$ is a disjoint union of open subsets of $\R^m$. A morphism from $A$ to $X$ is an embedding $f \colon A\hookrightarrow X$ having the property that the restriction of $f$ to each connected component of $A$ is a standard embedding into a component of $X$. We call such maps standard embeddings of $A$ into $X$. The space of standard embedding of $A$ into $X$ will be denoted $\sEmb(A, X)$. It is easy to see that a composition of standard embeddings is again a standard embedding, and so we have a topologically enriched category.

\end{definition}
\begin{definition}
Let $D^m$ be the unit open ball in $\R^m$. The $m$-dimensional little disks operad, denoted by $\balls_m$, is defined as follows. As a symmetric sequence, $\balls_m(k):=\sEmb(k\times D^m, D^m)$. As an operad, $\balls_m$ is the {\it endomorphism operad} of the unit disc $D^m$ in the category $\mfld$ (Definition~\ref{def: standard embeddings}) which is viewed as a symmetric monoidal category with disjoint union as symmetric product and empty set as unit.
\end{definition}
The operad $\balls_m$ is sometimes called the {\it unframed} little disks operad. We also will have some use for the framed little disks operad. There are slight variations in the literature regarding the definition of the framed little disks operad. Here is our definition.
\begin{definition}
Let $D^m$ be the unit ball in $\R^m$ and let $M$ be an open subset of $\R^m$. A {\it framed} embedding of $D^m$ into $M$ is an embedding that is the composition of an orthogonal transformation and a standard embedding. A framed embedding of a disjoint union of some copies of $D^m$ into $M$ is an embedding that resricts to a framed embedding on each copy of $D^m$. Let $\sEmb_\framed(A\times D^m, M)$ denote the space of framed embeddings of $A\times D^m$ into $M$, where $A$ is, as usual, a finite set. 
\end{definition}
\begin{definition}
The framed $m$-dimensional little disks operad is the operad whose $k$-th space is $\sEmb_\framed(k\times D^m, D^m)$. The operad structure maps are given by compositions, as usual. The framed little disks operad will be denoted by $\balls_m^\framed$. We also introduce the operad $\balls_m^\smooth$, whose $k$-th space is the space of all smooth embeddings $\Emb(k\times D^m, D^m)$.
\end{definition}
\begin{remark}
There are maps of operads $$\balls_m\longrightarrow \balls^\framed_m \longrightarrow \balls^\smooth_m.$$ Both maps are levelwise inclusions, and the second map is a levelwise  homotopy equivalence.
\end{remark}
Our main goal in this section is to describe the right modules, and especially the infinitesimal bimodules over $\balls_m$, using the theory of the previous section. Our first task is to describe explicitly the categories $\calF(\balls_m)$ (we also describe $\calF(\balls^\framed_m)$ and $\calF(\balls^\smooth_m)$ along the way) and especially $\widetilde\Gamma(\balls_m)$ (but not $\widetilde\Gamma(\balls^\framed_m)$ or $\widetilde\Gamma(\balls^\smooth_m)$). We begin with the category $\calF(\balls_m)$ and the associated category of right modules over $\balls_m$. The following lemma is an immediate consequence of the definitions.
\begin{lemma}
The category $\calF(\balls_m)$ can be identified with the full subcategory of $\mfld$ whose objects are disjoint unions of copies of $D^m$. More explicitly, it is the topological category whose objects are finite sets (or, equivalently, finite disjoint unions of copies of $D^m$) and where the space of morphisms from $A$ to $B$ is $\sEmb(A\times D^m, B\times D^m)$. 

Similarly, the category $\calF(\balls^\framed_m)$ (resp. $\calF(\balls^\smooth_m)$) is the topological category whose objects are finite sets (or, equivalently, finite disjoint unions of copies of $D^m$) and where the space of morphisms from $A$ to $B$ is $\sEmb_\framed(A\times D^m, B\times D^m)$ (resp. $\Emb(A\times D^m, B\times D^m)$). 
\end{lemma}
We will refer to $\calF(\balls_m)$ as the category of disjoint unions of standard balls and standard embeddings between them. The following lemma is a special case of Lemma~\ref{lemma: right module}.
It certainly is well known, but we include it here as a warm-up for the analogous statements about infinitesimal bimodules.
\begin{lemma}\label{lem: right}
The category of right modules over $\balls_m$ with values in $\Top$ is equivalent to the category of contravariant topological functors from $\calF(\balls_m)$ to $\Top$. There are analogous statements for right modules over $\balls^\framed_m$ and $\balls^\smooth_m$.
\end{lemma}
\begin{example}
Let $M$ be an open subspace of $\R^m$. We associate with $M$ a right module $\sEmb(-, M)$ defined by $\sEmb(-, M)(A):=\sEmb(A\times D^m, M)$. Obviously, this defines a contravariant functor from $\calF(\balls_m)$ to $\Top$ and thus a right module over the little disks operad. Similarly define right modules $\sEmb_\framed(-, M)$ and $\Emb(-, M)$ over $\balls_m^\framed$ and $\balls_m^\smooth$ respectively.
\end{example}
Our next, and main, task is perform a similar analysis of infinitesimal bimodules over $\balls_m$. Note that we do not consider infinitesimal bimodules over $\balls_m^\framed$ in this paper. 

We need to describe the category $\widetilde\Gamma(\balls_m)$. It follows easily from the definition that  $\widetilde\Gamma(\balls_m)$ is a category whose objects are pointed finite sets and where the space of morphisms from $S$ to $T$ (where $S$ and $T$ are pointed finite sets) is the space of standard embeddings of $S\times D^m$ into $T\times D^m$ that take the basepoint component into the basepoint component. However, composition in $\widetilde\Gamma(\balls_m)$ is not the standard one. 
We would like to present an alternative description of the category~$\widetilde\Gamma(\balls)$, in which composition is given by ordinary composition of functions. We achieve this by replacing the basepoint component $*\times D^m$ with an ``antiball''. 

Let $\overline{D^m}$ be the closed unit ball in $\R^m$. We refer to the complement $R^m\setminus \overline{D^m}$ as the antiball. For a pointed set $S$, define
\[
S\boxtimes D^m=(S\setminus\{*\})\times D^m \coprod (R^m\setminus \overline{D^m}).
\]
Thus $S\boxtimes D^m$ is the disjoint union of $|S|-1$ balls and one antiball.

The following proposition is the main result of this seciton.
\begin{proposition}\label{prop: description}
The category $\widetilde\Gamma(\balls_m)$ is equivalent to the following category. Its objects are pointed finite sets. Given two pointed sets $S, T$, the space of morphisms from $S$ to $T$ is \[
\sEmb(S\boxtimes D^m, T\boxtimes D^m).
\] 
Composition in this category is just ordinary composition of standard embeddings.
\end{proposition}
\begin{proof}
We need to show that two categories are equivalent. By definition, they have the same set of objects: finite pointed sets. We will construct homeomorphisms between corresponding mapping spaces, and then check that the homeomorphisms preserve composition.

Let $S, T$ be pointed sets. By definition, the space of maps from $S$ to $T$ in $\widetilde \Gamma(\balls_m)$ is
\[
\coprod_{\alpha\in \Gamma(S, T)} \prod_{t\in T} \sEmb(\alpha^{-1}(t)\times D^m, D^m).
\]
To compare this space with $\sEmb(S\boxtimes D^m, T\boxtimes D^m)$, note that there are no standard morphisms from the antiball to the unit ball. Therefore we may write this space as
\[
\coprod_{\alpha\in \Gamma(S, T)} \sEmb\left(\alpha^{-1}(*)\boxtimes D^m, \R^m\setminus\overline{D^m}\right)\times\prod_{t\in T\setminus\{*\}} \sEmb(\alpha^{-1}(t)\times D^m, D^m).
\]
Therefore, we need to establish a homeomorphism, where $U$ is a pointed set ($U=\alpha^{-1}(*)$ in the above formula)
\begin{equation}\label{eq: embeddings}
\sEmb(U\times D^m, D^m)\cong \sEmb(U\boxtimes D^m, \R^m\setminus \overline{D^m}).
\end{equation}
Let $\alpha\colon D^m \hookrightarrow D^m$ be a standard embedding. There exists a unique standard isomorphism $\bar{\alpha}\colon \R^m\to \R^m$, such that $\alpha$ is the restriction of $\bar{\alpha}$ to the unit ball $D^m$. $\bar{\alpha}$ has an inverse. By a slight abuse of notation, we denote this inverse simply by $\alpha^{-1}$. We will also denote by $\alpha^{-1}$ the restriction of $\alpha^{-1}$ to any subset of $\R^m$. Note that  ${\alpha}^{-1}$ takes the antiball into the antiball. 

Now let $\alpha\colon U\times D^m\hookrightarrow D^m$ be a standard embedding of a union of balls. Here $U$ is a pointed finite set. Let's say that $U=\{*, u_1, \ldots, u_n\}$. Let $\alpha_*, \alpha_1, \ldots, \alpha_n$ be the restrictions of $\alpha$ to $*\times D^m$, $u_1\times D^m$, etc. Then it is easy to see that $(\alpha_*^{-1}, \alpha_*^{-1}\alpha_1, \ldots, \alpha_*^{-1}\alpha_n)$ defines a standard embedding of $U\boxtimes D^m\hookrightarrow \R^m\setminus\overline{D^m}$. The assignment $$(\alpha_*, \alpha_1, \ldots, \alpha_n)\mapsto (\alpha_*^{-1}, \alpha_*^{-1}\alpha_1, \ldots, \alpha_*^{-1}\alpha_n)$$ defines the homemorphism~\eqref{eq: embeddings}

It remains to show that the homeomorphism takes composition in $\widetilde\Gamma(\balls_m)$ to ordinary composition of embeddings. To do this, we will first describe explicitly the composition law in $\widetilde \Gamma(\balls_m)$.

Let $\alpha\colon S\times D^m \to T\times D^m$ and $\beta\colon T\times D^m \to U\times D^m$ be standard embeddings (taking the basepoint component into the basement component), considered as morphisms in $\widetilde\Gamma(\balls_m)$. Let $\beta\tilde\circ\alpha\colon S\times D^m \to U\times D^m$ be the composition of $\beta$ and $\alpha$ in $\widetilde\Gamma(\balls_m)$, while $\beta\alpha$ denotes the standard composition. In what follows, let $D^m_*=\{*\}\times D^m$ be the basepoint component of $S\times D^m, T\times D^m$, and $U\times D^m$ (keep in mind that $*$ is the common basepoint of all pointed sets). To describe $\beta\tilde\circ\alpha$ in terms of standard compositions of embeddings, we will write $S\times D^m$ as a disjoint union of several subsets: Let $S_*=D^m_*$, $S_1=\alpha^{-1}(D^m_*)\setminus D^m_*$, $S_2=(\beta\alpha)^{-1}(D^m_*)\setminus \alpha^{-1}(D^m_*)$, and $S_3=(\beta\alpha)^{-1}(U\times D^m\setminus D^m_*).$ Clearly, $S\times D^m$ is the disjoint union of $S_*$, $S_1$, $S_2$, and $S_3$. Let $\alpha_*, \alpha_1, \alpha_2$, and $\alpha_3$ be the restrictions of $\alpha$ to $S_*, S_1, S_2$, and $S_3$ respectively. We write $\alpha=(\alpha_*, \alpha_1, \alpha_2, \alpha_3)$.

Similarly, we partition $T\times D^m$ as follows. Let $T_*=D^m_*$, $T_1=\beta^{-1}(D^m_*)\setminus D^m_*$, and $T_2=\beta^{-1}(U\times D^m\setminus D^m_*)$. Clearly, $T\times D^m=T_*\coprod T_1\coprod T_2$. Let $\beta_*, \beta_1, \beta_2$ be the restrictions of $\beta$ to $T_*, T_1,$ and $T_2$ respectively. Unraveling the definition of $\widetilde\Gamma(\balls_m)$ (definition~\ref{def: twisted}), one finds that 
$$\beta\tilde\circ \alpha=(\alpha_*\beta_*, \alpha_1, \alpha_*\beta_1\alpha_2,\beta_2\alpha_3)$$ where the four components denote the restriction of $\beta\tilde\circ\alpha$ to $D^m_*, S_1, S_2$, and $S_3$ respectively.

Now let us examine the images of $\alpha, \beta,$ and $\beta\tilde\circ\alpha$ in $\sEmb(S\boxtimes D^m, T\boxtimes D^m)$, $\sEmb(T\boxtimes D^m, U\boxtimes D^m)$, and $\sEmb(S\boxtimes D^m, U\boxtimes D^m)$ respectively. By definition, $\alpha=(\alpha_*, \alpha_1, \alpha_2, \alpha_3)$ goes to $(\alpha_*^{-1}, \alpha_*^{-1}\alpha_1, \alpha_2, \alpha_3)$. Similarly $\beta=(\beta_*, \beta_1, \beta_2)$ goes to $(\beta_*^{-1}, \beta_*^{-1}\beta_1, \beta_2)$. Finally, $\beta\tilde\circ \alpha$ goes to $$(\beta_*^{-1}\alpha_*^{-1}, \beta_*^{-1}\alpha_*^{-1}\alpha_1, \beta_*^{-1}\beta_1\alpha_2,\beta_2\alpha_3).$$
Clearly, the image of  $\beta\tilde\circ \alpha$ is the composition of the image of $\beta$ and the image of $\alpha$, which completes the proof.

\end{proof}

\begin{examples}\label{ex: balls}
We conclude the section with some examples of contravariant functors on $\widetilde\Gamma(\balls_m)$, or equivalently of infinitesimal bimodules over $\balls_m$.

 As with every operad, $\balls_m$ is an infinitesimal bimodule over itself. By Proposition~\ref{prop: infinitesimal bimodule}, there is a corresponding contravariant functor from $\widetilde\Gamma(\balls_m)$ to $\Top$. It is defined on objects by the formula
$$T\mapsto \sEmb\left((T\setminus\{*\})\times D^m, D^m\right).$$ To understand the functoriality with respect to morphisms in $\widetilde\Gamma(\balls_m)$, let us first notice that $\sEmb\left((T\setminus\{*\})\times D^m, D^m\right)$ can be identified with the subspace of $\sEmb\left(T\boxtimes D^m, \R^m\right)$ consisting of those standard embeddings that restrict to the inclusion on the antiball. Let $f$ be such an embedding, and let
$$\alpha\colon S\boxtimes D^m \to T\boxtimes D^m$$ be a morphism in $\widetilde\Gamma(\balls_m)$. Then $f\circ \alpha$ is a standard embedding of $S\boxtimes D^m$ into $\R^m$. $f\circ\alpha$ may not restrict to the inclusion on $\R^m\setminus \overline{D^m}$. Let $\alpha_*\colon \R^m\setminus \overline{D^m}\longrightarrow \R^m\setminus \overline{D^m}$ be the restriction of $\alpha$ to the antiball. As in the proof of Proposition~\ref{prop: description}, let $\alpha_*^{-1}$ be the inverse of the isomorphism of $\R^m$ determined by $\alpha$. It is easy to see that $\alpha_*^{-1} f\alpha$ is an embedding of $S\boxtimes D^m$ into $\R^m$ that restricts to the identity on the antiball. The functoriality of $\sEmb\left((T\setminus\{*\})\times D^m, D^m\right)$ is defined by the formula $f\mapsto \alpha_*^{-1} f\alpha$. It is easy to check that this is the correct definition.

A perhaps even more natural example of a contravariant functor from $\widetilde\Gamma(\balls_m)$ to $\Top$ (and thus of an infinitesimal bimodule over $\balls_m$) is the functor 
$$S\mapsto \sEmb\left(S\boxtimes D^m, \R^m\right).$$
This is obviously a contravariant functor $\widetilde\Gamma(\balls_m)\longrightarrow \Top$. In fact, it is weakly equivalent to the previous functor.
\begin{lemma}\label{lemma: equivalent infinitesimal bimodules}
There is a natural transformation of functors
\begin{equation}\label{eq: projection}
\sEmb\left(S\boxtimes D^m, \R^m\right)\longrightarrow \sEmb\left(S\setminus\{*\}\times D^m, D^m\right)
\end{equation}
which is a homotopy equivalence for each $S$.
\end{lemma}
\begin{proof}
As before, let us identify $\sEmb\left((S\setminus\{*\})\times D^m, D^m\right)$ with the subspace of $$\sEmb\left(S\boxtimes D^m, \R^m\right)$$ consisting of those embeddings that restrict to the inclusion on $\R^m\setminus \overline D^m$. For an $f\in \sEmb\left(S\boxtimes D^m, \R^m\right),$ let us write $f=(f_*, f_1)$, where $f_*$ is the restriction of $f$ to the antiball, and $f_1$ is the restriction of $f$ to $(S\setminus\{*\})\times D^m$.
We define the map~\eqref{eq: projection} by the formula $f\mapsto f_*^{-1}f_1$.
 It is easy to check that it is natural with respect to morphisms in $\widetilde\Gamma(\balls_m)$, and that it is a homotopy equivalence. In fact, it is a fiber bundle with contractible fibers (the fiber is the space of standard isomorphisms of $\R^m$). 
\end{proof}
\begin{remark}\label{remark: equivalent infinitesimal bimodules}
The lemma can be interpreted as follows. The functor
$$S\mapsto \sEmb\left(S\boxtimes D^m, \R^m\right)$$ is equivalent to $\balls_m$ as an infinitesimal bimodule over $\balls_m$. If $\R^m\subset \R^n$ then it follows that $ \sEmb\left(S\boxtimes D^n, \R^n\right)$ is equivalent to $\balls_n$ as an infinitesimal bimodule over $\balls_n$, and therefore also over $\balls_m$.
\end{remark}

These infinitesimal bimodules are important to us, because they give a model for the functor $U\mapsto \Ebarc(U, \R^n)$. We will clarify this point in the next section.
\end{examples}

\section{Taylor tower as the space of module maps}\label{section: module maps}
We will now recall the basic setup of M. Weiss's embedding calculus (also known as manifold calculus)~\cite{WeissEmb, DeBrito-Weiss}. 
For a topological space $X$, let $\calO(X)$ be the poset/category of open subsets of $X$.
Let $M$ be an $m$-dimensional manifold. Manifold calculus is concerned with presheaves on $M$. In other words, with contravariant functors from $\calO(M)$ to $\Top$ (or more generally, to a Quillen model category). Following Weiss, we say that a functor 
$$F\colon \calO(M)^{\operatorname{op}}\longrightarrow \calD$$
is {\it good} if (a) $F$ takes isotopy equivalences to weak homotopy equivalences and (b) for any sequence $U_0\subset U_1 \subset \cdots$ of open subsets of $M$, whose union is $U$, the following natural map is an equivalence
\[
F(U)\longrightarrow \underset{i}{\holim} F(U_i).
\]
We need to consider a ``compactly supported'' version of manifold calculus for subsets of $\R^m$, for the study of functors that are invariant with respect to isotopies with bounded support. This is similar to calculus for manifolds with boundary (see~\cite[Section 10]{WeissEmb} and~\cite[Section 9]{DeBrito-Weiss}).
\begin{definition}
Let $\widetilde\calO(\R^m)$ be the poset/category of open subsets of $\R^m$ whose complement is bounded.
\end{definition}
We say that a morphism $U\hookrightarrow V$ in $\widetilde\calO(\R^m)$ is an isotopy equivalence if there is a smooth embedding $V\hookrightarrow U$, that coincides with the identity outside a bounded subset of $V$, such that both compositions are isotopic to the identity via an isotopy that is constant outside a bounded subset of $U$ or $V$, as appropriate. As before, we say that a functor $F\colon \widetilde\calO(\R^m)^{\operatorname{op}}\longrightarrow \calD$ is good if it converts isotopy equivalences to weak homotopy equivalences and filtered unions to homotopy limits.

Let $F$ be  a good contravariant functor from either $\calO(M)$ or $\widetilde\calO(\R^m)$ to a model category $\calD$. We say that $F$ is polynomial of degree $k$, if it takes strongly co-cartesian $k+1$-dimensional cubical diagrams to homotopy cartesian cubical diagrams (this is equivalent to~\cite[Definition 2.2]{WeissEmb}).

Weiss proves that good functors can be approximated by polynomial functors.
Let $F$ be a good contravariant functor either from $\calO(M)$ or $\widetilde\calO(\R^m)$ to $\calD$. For each $k\ge 0$, there is a polynomial functor of degree $k$, which we denote $\ET_kF$, and a natural transformation $F\to \ET_kF$ that is initial (in the homotopy category of functors) among maps from $F$ to a polynomial functor of degree $k$. Moreover, the natural transformation $F\to \ET_kF$ induces an equivalence when evaluated on certain open subsets of $M$ or $\R^m$, and this property characterizes $\ET_kF$.
\begin{definition}
 Let $M$ be a smooth $m$-dimensional manifold. For each $k\ge 0$, define $\calO_k(M)\subset \calO(M)$ to be the subposet consisting of images of smooth embeddings $A\times D^m\hookrightarrow M$, where $A$ is a set with at most $k$ elements. Also define $\widetilde\calO_k(\R^m)\subset \widetilde\calO(\R^m)$ to be the subposet consisting of images of compactly supported smooth embeddings $S\boxtimes D^m\hookrightarrow \R^m$. Here $S$ is a pointed set with at most $k$ non-basepoint elements. Recall that $S\boxtimes D^m$ is the disjoint union of $|S|-1$ balls and one antiball. ``Compactly supported'' means that the embedding $S\boxtimes D^m\hookrightarrow \R^m$ is required to agree with the identity on the antiball outside a bounded set.
\end{definition}
Weiss proves the following theorem. The case of $\calO(M)$ follows from~\cite[Theorems 5.1 and 6.1]{WeissEmb}. The case of $\widetilde\calO(\R^m)$ is equivalent to the case of manifolds with boundary, discussed in~\cite[Section 10]{WeissEmb} and~\cite[Section 9]{DeBrito-Weiss}. 
\begin{theorem}\label{theorem: characterization}
Let $F$ be a good contravariant functor from $\calO(M)$ (resp. $\widetilde\calO(\R^m)$). Then the map $F(U)\longrightarrow \ET_k F(U)$ is a weak equivalence for $U\in\calO_k(M)$ (resp. $U\in\widetilde\calO_k(\R^m)$). Moreover, $\ET_kF$ is uniquely characterized (up to natural equivalence) as the degree $k$ polynomial functor with this property.
\end{theorem}
In fact, $\ET_kF(M)$ (resp. $\ET_kF(\R^m)$) is defined as the homotopy limit of $F(U)$, as $U$ ranges over $\calO_k(M)$ (resp. $\widetilde\calO_k(\R^m)$). In other words, $\ET_kF$ is obtained by restricting $F$ from $\calO(M)$ to $\calO_k(M)$ and then taking derived right Kan extension back to $\calO(M)$.

We will show that with some additional assumptions on $F$ and $M$ one can obtain an ``operadic'' formula for $\ET_kF(M)$. More specifically, we will show that if $M$ is an open subset of $\R^m$, and $F$ is ``context free'', one can express $\ET_kF$ in terms of modules over the little disks operad. 
\begin{remark} Such a reduction can be achieved more generally, whenever $M$ is parallelizable. To achieve this, it might be convenient to replace the little disks operad with the Fulton-McPherson operad, as in~\cite{Turchin13}.
\end{remark}

Fix a dimension $m$. Recall that $\mfld$ is the topologically enriched category whose objects are disjoint union of open subsets of $\R^m$ and whose morphisms are spaces of standard embeddings. We will need a ``compactly supported'' version of this category.


\begin{definition}
Let $\widetilde\mfld$ be the following topologically enriched category. An object of $\widetilde\mfld$ is a pair $(U, U_0)$, where $U$ is a disjoint union of open subsets of $\R^m$, and $U_0$ is a connected component of $U$ that  is the complement of a compact subset of $\R^m$. We refer to $U_0$ as the marked component of $U$. Let $(U, U_0)$ and $(V, V_0)$ be two objects of $\widetilde\mfld$. Morphisms from $(U, U_0)$ to $(V, V_0)$ in $\widetilde\mfld$ are standard embeddings from $U$ to $V$ that take $U_0$ into $V_0$. We denote the space of morphisms from $(U, U_0)$ to $(V, V_0)$ by $\sEmb_0(U, V)$.
\end{definition}
The category $\widetilde\mfld$ is analogous to category $\mathrm{Man}^\partial$ of manifolds with prescribed boundary considered in~\cite[Section 9]{DeBrito-Weiss}. 

Suppose $M$ is an open subset of $\R^m$. Then there are obvious ``inclusion'' functors $\calO(M)\longrightarrow \mfld$ and $\widetilde\calO(\R^m)\longrightarrow \widetilde\mfld$.
\begin{definition}
Let $F$ be a good contravariant functor from $\calO(M)$ (resp. from $\widetilde\calO(\R^m)$) to a topologically enriched model category $\calD$. We say that $F$ is {\em context-free} if it factors (up to natural equivalence) through the category $\mfld$ (resp. $\widetilde\mfld$). 
\end{definition}
To be more explicit, for example in the case when the domain of $F$ is $\calO(M)$, we require that there is a continuous functor $F'\colon \mfld^{\operatorname{op}} \longrightarrow \calD$ such that $F$ is weakly equivalent to the composed functor
\[
\calO(M)^{\operatorname{op}}\longrightarrow\mfld^{\operatorname{op}}\stackrel{F'}{\longrightarrow} \calD.
\]
If $F$ is isomorphic to the composed functor, then we say that $F$ is {\it strictly context-free}.
\begin{remark}
Our definition of ``context-free'' differs slightly from the definition given by de Brito-Weiss~\cite{DeBrito-Weiss} or Turchin~\cite{Turchin13}. Roughly speaking the difference is that those works require a context-free functor to be defined on the category of all smooth manifolds and codimension-zero embeddings, while we are working with what amounts to a category of manifolds with a chosen trivialization of the tangent bundle, and smooth embeddings that respect the trivialization.
\end{remark}
\begin{example}\label{rem: ebar}
Suppose that $M$ is an open subset of $\R^m$, and $\R^m$ is a linear subspace of $\R^n$. Then one may define a functor 
$$\Ebar(-, \R^n)\colon \calO(M)^{\operatorname{op}}\longrightarrow \Top.$$
Recall that $\Ebar(U, \R^n)$ is the homotopy fiber of the map $\Emb(U, \R^n)\longrightarrow \Imm(U, \R^n).$ The definition requires the presence of a basepoint in $\Imm(U, \R^n)$ (or in $\Emb(U, \R^n)$, if one wants the functor to take values in pointed spaces). The basepoint is provided by the inclusion of $M$ into $\R^n$, via $\R^m$. This means that the functor is not strictly context-free, because the basepoint of $\Imm(U, \R^n)$ depends on the inclusion of $U$ into $M$. However, one can show that it is context-free using Smale-Hirsch theory. Let $\inj(\R^m, \R^n)$ be the space of injective linear maps from $\R^m$ to $\R^n$. There is a natural map, given by differentiation
$$\Imm(U, \R^n) \longrightarrow \map(U, \inj(\R^m, \R^n)).$$
By Smale-Hirsch theory, this map is an equivalence if $m+1\le n$. Furthermore, let $\widetilde\inj(\R^m, \R^n)$ be the quotient of $\inj(\R^m, \R^n)$ by the group of positive reals acting by multiplication. We have a homotopy equivalence
$$\Imm(U, \R^n) \stackrel{\simeq}{\longrightarrow} \map(U, \widetilde\inj(\R^m, \R^n)).$$
The space $\map(U, \widetilde\inj(\R^m, \R^n))$ has a canonical basepoint: it is the constant map that sends $U$ to the equivalence class of the fixed linear inclusion of $\R^m$ into $\R^n$. Moreover, for a standard embedding $U \hookrightarrow V$, the induced map
$$\map(V, \widetilde\inj(\R^m, \R^n)) \longrightarrow \map(U, \widetilde\inj(\R^m, \R^n))$$
preserves the basepoint. Therefore the functor
$$U\mapsto  \map(U, \widetilde\inj(\R^m, \R^n))$$
is a context-free functor with values in pointed spaces. Using it as our model for the immersions functor, we obtain that the functor $\Ebar(-, \R^n)$ is context-free, as a functor to spaces (but not as a functor to pointed spaces).

For another example, consider the functor
$$\Ebar_\st(-, \R^n)\colon \widetilde\calO(\R^m) \longrightarrow \Top.$$ Here $\Ebar_\st(U, \R^n)$ is defined to be the homotopy fiber of the map $\Emb_\st(U, \R^n)\longrightarrow \Imm_\st(U, R^n),$ where $\Emb_\st(U, \R^n)$ is the space of smooth embeddings of $U$ into $\R^n$ that agree, outside a bounded set, with a standard embedding into $\R^m$, followed by the inclusion $\R^m\hookrightarrow \R^n$. The space of immersions $\Imm_\st(U, R^n)$ is defined analogously. Again, this functor is context free, although not in the strict sense. This is because $\Imm_\st(U, \R^n)$ is naturally equivalent to the space of pointed maps $$\map_*(U, \widetilde\inj(\R^m, \R^n)).$$ Here by ``pointed'' we mean maps that send everything outside a bounded subset of $U$ to the basepoint in $\widetilde\inj(\R^m, \R^n)$. The space of pointed maps has a canonical basepoint. It is easy to see that the functor $\map_*(-, \widetilde\inj(\R^m, \R^n))$ can be extended to a functor from $\widetilde\mfld$ to pointed spaces, and therefore is context-free. The functor $\Ebar_\st(-, \R^n)$ is equivalent to the homotopy fiber of the map $\Emb_\st(-, \R^n) \longrightarrow \map_*(-, \widetilde\inj(\R^m, \R^n))$, and thus it also is context-free.
\end{example}

Recall that $\calF(\balls_m)$ is the full subcategory of $\mfld$ whose objects are finite disjoint unions of copies of the unit ball in $\R^m$. For $k\ge 0$, let $\calF_{\le k}(\balls_m)\subset \calF(\balls_m)$ be the subcategory of unions of at most $k$ balls. Similarly $\widetilde\Gamma(\balls_m)$ is the full subcategory of $\widetilde\mfld$ whose objects are a finite union of standard balls and one antiball. The antiball will be the marked component. Let $\widetilde\Gamma_{\le k}(\balls_m)$ be the full subcategory consisting of unions of at most $k$ balls and the antiball.

Suppose $F\colon \calO(M)\longrightarrow \Top$ is a context-free contravariant functor. Then $F$ gives rise to a contravariant functor from $\calF(\balls_m)$ to $\Top$, which we denote with the same letter $F$. By Lemma~\ref{lem: right}, $F$ can be thought of as a right module over the operad $\balls_m$. Similarly, a context-free contravariant functor from $\widetilde\calO(\R^m)$ to $\Top$ can be identified with an infinitesimal bimodule over $\balls_m$ by Propositions~\ref{prop: infinitesimal bimodule} and~\ref{prop: description}.

Given an operad $O$ and right modules $P$ and $Q$ over $O$, let $\underset{O}{\hRmod}(P, Q)$ denote the derived space of morphisms from $P$ to $Q$. It is the space of morphisms from a cofibrant replacement of $P$ to a fibrant replacement of $Q$ (we refer here to the projective model structure on the category of modules, in which weak equivalences and fibrations are defined degree-wise). Let $\underset{O}{\hRmod_{\le k}}(P, Q)$ denote the derived morphism space of modules truncated at $k$. It can be thought of as the space of derived natural transformations of functors on the category $\calF_{\le k}(O)$ - the full subcategory of $\calF(O)$ consisting of sets of cardinality at most $k$. Similarly, if $P$ and $Q$ are infinitesimal bimodules over $O$, we let $\underset{O}{\hWBimod}(P, Q)$ denote the derived space of infinitesimal bimodule maps from $P$ to $Q$, and let $\underset{O}{\hWBimod_{\le k}}(P, Q)$ denote the derived mapping space of $k$-truncated infinitesimal bimodules. Again, note that it can be thought of as a space of derived natural transformations between functors on $\widetilde\Gamma_{\le k}(\balls_m)$.

%
%

The following theorem is similar to one of de Brito-Weiss~\cite{DeBrito-Weiss}, and the proof is similar too. The difference is that they work with all manifolds and therefor their result involves the framed little disks operad. We restrict ourselves to the category of manifolds with a chosen trivialization, and therefore our statement is phrased in terms of the unframed little disks operad.
\begin{proposition}\label{prop: context-free}
Let $M$ be an open subset of $\R^m$. Suppose $F$ is a good contravariant functor from $\calO(M)$ to a model category $\calD$ that is enriched, tensored and cotensored over topological spaces. Moreover, suppose that $F$ is context-free, so we may think of $F$ as a right module over $\balls_m$. Then there is a natural equivalence for all $k\le \infty$
\[
\ET_kF(U)\simeq \underset{\balls_m}{\hRmod_{\le k}}\left(\sEmb(-, U), F(-)\right).
\]
Similarly, if $F$ is a good, contravariant and context-free functor from $\widetilde\calO(\R^m)$ to $\calD$, then there is an equivalence for $k\le \infty$
\[
\ET_kF(U)\simeq \underset{\balls_m}{\hWBimod_{\le k}}\left(\sEmb(-, U), F(-)\right).
\]
\end{proposition}
\begin{proof}
First we consider the case when $F$ is a good functor on $\calO(M)$. Let $V$ be a fixed object of $\calF_{\le k}(\balls_m)$, and suppose $V$ is the disjoint union of $i$ balls. Consider the covariant functor from $\mfld$ to $\Top$ given by the formula $U\mapsto \sEmb(V, U)$. Notice that it is naturally equivalent to the functor that associates to $U$ the configuration space of ordered $i$-tuples of points in $U$. It follows that it is an isotopy functor in the sense that if $U \longrightarrow U_1$ is a standard embedding that happens to be an isotopy equivalence, then the induced map $ \Emb(V, U)\longrightarrow  \Emb(V, U_1)$ is a homotopy equivalence. It also takes filtered unions to filtered homotopy colimits. Furthermore, it is easy to show that this functor of $U$ takes $i+1$-dimensional strongly cocartesian cubical diagrams to homotopy cocartesian cubical diagrams. It follows that for any space $Y$, the functor $U\mapsto \map(\Emb(V, U), Y)$ is a good functor that is polynomial of degree $k$. Recall that the derived space of module maps is a derived space of natural transformations. By a well-known construction, our space of derived natural transformation can be presented as a homotopy inverse limit of mapping spaces, which in our case have the form $ \map(\Emb(V, U), Y)$. Therefore our construction of $\ET_kF(U)$ is a homotopy limit of good functors of degree $k$. It follows that (our construction of) $\ET_kF$ is itself a good functor that is polynomial of degree $k$. Furthermore, there is a tautological transformation from $F$ to our construction of $\ET_kF$. Note that if $U$ happens to be a union of at most $k$ balls, or in other words an object of $\calF_{\le k}(\balls_m)$, then $\Emb(-, U)$ is a representable contravariant functor from $\calF_{\le k}(\balls_m)$ to $\Top$. By enriched Yoneda lemma, the transformation $F(U)\longrightarrow \ET_kF(U)$ induces an equivalence when $U$ is the union of at most $k$ balls. It now follows from Theorem~\ref{theorem: characterization} that our construction of $\ET_kF$ is indeed the $k$-th Taylor approximation of $F$.

The case when $F$ is a functor on $\widetilde\calO(\R^m)$ is done similarly. The point to note is this. Suppose that $V$ is an object of $\widetilde\mfld_k$. In other words, $V$ is the disjoint union of $i$ balls and an antiball. Let $U$ be an object of $\widetilde\calO(\R^m)$. Then $\sEmb(V, U)$ is still naturally equivalent to the configuration space of $i$-tuples of points in $U$. The rest of the argument proceeds in the same way as before.
\end{proof}
We are now ready to state the main result of this section, which is a description of the Taylor tower of the functors $\Ebar(M, \R^n)$ and $\Ebar_\const(\R^m, \R^n)$ in terms of module maps over $\balls_m$.

As before, fix a linear inclusion $\R^m\hookrightarrow \R^n$, and assume that $m< n$. This induces a map of operads $\balls_m\to \balls_n$, and in particular endows $\balls_n$ with the structure of an infinitesimal bimodule over $\balls_m$ and also of a right module over $\balls_m$. 

Let $M$ be an open subset of $\R^m$. Recall that $\sEmb(-, M)$ is a right module over $\balls_m$.
%

\begin{theorem}\label{theorem: operadic}
Let $M$ be an open subset of $\R^m$ and suppose $\R^m$ is a linear subspace of $\R^n$, with $m<n$. For all $k\le \infty$, there are equivalences
\[
\ET_k\Ebar(M, \R^n)\simeq \underset{\balls_m}{\hRmod_{\le k}}(\sEmb(-, M), \balls_n)
\]
\[
\ET_k\Ebar_\const(\R^m, \R^n)\simeq \underset{\balls_m}{\hWBimod_{\le k}}(\balls_m, \balls_n).
\]
\end{theorem}
\begin{proof}
Let us first consider the case of $\Ebar(M, \R^n)$. This case is a nearly immediate consequence of Proposition~\ref{prop: context-free}. We saw in Example~\ref{rem: ebar} that the functor $\Ebar(-, \R^n)$ is context-free. It is not hard to see that $\Ebar(-, \R^n)$ is equivalent to $\sEmb(-, \R^n)$ as a right module over $\balls_m$. In fact, this claim is proved as~\cite[Proposition 7.1]{ALV}. On the other hand, $\sEmb(-, \R^n)$ is equivalent to $\balls_n$ as a right module over $\balls_m$ (remark~\ref{remark: equivalent infinitesimal bimodules}).

Now let us consider the case of $\Ebar_\const(\R^m, \R^n)$. This case is proved similarly, starting with Proposition~\ref{prop: context-free}. The functor $U\mapsto \Ebar_c(U, \R^n)$ is not strictly context-free - it does not extend from $\widetilde\calO(\R^m)$ to $\widetilde \mfld$. On the other hand, there is a natural inclusion of functors on $\widetilde\calO(\R^m)$
$$
\Ebar_\const(-, \R^n) \longrightarrow \Ebar_\st(-, \R^n)
$$
where $\Ebar_\st(-, \R^n)$ is the functor introducted in Example~\ref{rem: ebar}. It is easy to show that this inclusion is a weak equivalence. In fact, for each $U$ the space $\Ebar_\st(U, \R^n)$ deformation retracts onto $\Ebar_\const(U, \R^n)$. As explained in Example~\ref{rem: ebar}, the functor $\Ebar_\st(-, \R^n)$ is context-free. It is easy to check that there is an equivalence of functors on $\widetilde\Gamma(\balls_m)$
$$
\Ebar_\st(-, \R^n) \simeq \sEmb(-, \R^n).
$$
Or in other words, $\Ebar_\st(-, \R^n)$ and $\sEmb(-, \R^n)$ are equivalent as infinitesimal bimodules over $\balls_m$. On the other hand, $\sEmb(-, \R^n)$ is equivalent to $\balls_n$ as an infinitesimal bimodule over $\balls_m$ by Remark~\ref{remark: equivalent infinitesimal bimodules}. The theorem follows.

\end{proof}
\begin{remark}
Another proof of the second part of the theorem is given in~\cite{Turchin13}, using a different approach. In that paper, the Fulton-McPherson operad is used in place of the little disks operad.
\end{remark}
\begin{remark}
The operad $\balls_1$ is weakly equivalent to the associative operad. It is not difficult to show that infinitesimal bimodules over the associative operad (considered as a non-$\Sigma$-operad) are the same thing as cosimplicial spaces. This explains why there is a cosimplicial model for the Taylor tower of the space of long knots~\cite{Sinha}.
\end{remark}

In subsequent parts of the paper we will need  the following homological version of Theorem~\ref{theorem: operadic}.
\begin{proposition}\label{prop: chainlevel}
Let $M$ be an open subset of $\R^m$ and suppose $\R^m$ is a linear subspace of $\R^n$, with $m<n$. For all $k\le\infty$, there are equivalences
\[
\ET_k\chains(\Ebar(M, \R^n))\simeq \underset{\chains(\balls_m)}{\hRmod_{\le k}}(\chains(\sEmb(-, M)), \chains(\balls_n))
\]
\[
\ET_k\chains(\Ebar_\const(\R^m, \R^n))\simeq \underset{\chains(\balls_m)}{\hWBimod_{\le k}}(\chains(\balls_m), \chains(\balls_n)).
\]
\end{proposition}
Proposition~\ref{prop: chainlevel} is proved in the same way as Theorem~\ref{theorem: operadic}. Notice that the right hand sides of the formulas can be written as 
\[
\ET_k\chains(\Ebar(M, \R^n))\simeq \underset{\balls_m}{\hRmod_{\le k}}(\sEmb(-, M), \chains(\balls_n))
\]
and
\[
\ET_k\chains(\Ebar_\const(\R^m, \R^n))\simeq \underset{\balls_m}{\hWBimod_{\le k}}(\balls_m, \chains(\balls_n)).
\]

\section{Applying the formality theorem}\label{section: applying formality}
As usual, let $M$ be an open subspace of $\R^m$ and  let us suppose that $\R^m$ is a subspace of another Euclidean space $\R^n$ such that $n\ge 2m+1$. 
We saw in Proposition~\ref{prop: chainlevel} that the Taylor towers of $\chains(\Ebar(M, \R^n))$ and $\chains(\Ebar_\const(\R^m, \R^n))$ can be expressed in terms of maps of modules over the operad $\chains(\balls_m)$. In this section we use Kontsevich's formality theorem, very much in the spirit of~\cite{ALV}, to deduce that if one works over the reals then $\chains(\balls_n)$ can be replaced with $\HH(\balls_n)$ in the formulas of Proposition~\ref{prop: chainlevel} (and ultimately, the same conslusion holds over the rationals). 

Recall that $\Com$ is the commutative operad in any given symmetric monoidal category. Thus $\Com(i)$ is isomorphic to the unit object in every arity $i$. When we need to emphasize the background category, we write $\Com^{\Top}$ or $\Com^\Ch$. There is a trivial map of operads $\balls_m\longrightarrow \Com^{\Top}$, which induces a map of operads $\chains(\balls_m)\longrightarrow \Com^{\Ch}$. For $n\ge 2$, since all the spaces in the operad $\balls_n$ are connected, there is an isomorphism of operads $\Com^{\Ch} \cong \HH_0(\balls_n)$. Finally note that there is a map of operads $\HH_0(\balls_n)\longrightarrow \HH(\balls_n)$. Together these maps endow
$\HH(\balls_n)$ with the structure of an infinitesimal bimodule (and in particular a right module) over $\Com^{\Ch}$, and over $\chains(\balls_m)$. This structure is referred to in the statement of the following proposition.
\begin{proposition}\label{prop: first reduction}
Suppose that $2m+1\le n$. There are equivalences
$$
\ET_k\left(\chains^\Q(\Ebar(M, \R^n)) \right)\simeq \underset{\chains(\balls_m)}{\hRmod_{\le k}}(\chains(\sEmb(-, M)), \HH(\balls_n;\Q))$$
and
$$
\ET_k \left(\chains^\Q(\Ebarc(\R^m, \R^n)) \right) \simeq \underset{\chains(\balls_m)}{\hWBimod_{\le k}}(\chains(\balls_m), \HH(\balls_n;\Q)).
$$
\end{proposition}
\begin{proof}
The first assertion of the proposition is essentially the same as~\cite[Theorem 9.2]{ALV}. We will prove the second assertion. The proof of the first assertion is similar, but easier. The idea of our proof here is the same as in [op. cit.]. The main point is that since the category of infinitesimal bimodules over an operad is equivalent to a category of diagrams (Proposition~\ref{prop: infinitesimal bimodule}), just like the category of right modules, all the formal manipulations that were done in~\cite{ALV} with right modules can also be done with infinitesimal bimodules. 

The main ingredient of the proof is Kontsevich's formality theorem~\cite{Kontsevich}, or more precisely the ``relative'' version of the theorem, introduced by Lambrechts and Volic~\cite{LV}. This theorem is only known to hold over $\R$ rather than over $\Q$, so we will first prove the proposition over $\R$, and then conclude it over $\Q$. 

Relative formality says that under the assumption $2m+1\le n$, the map of operads
$$\chains^\R(\balls_m)\longrightarrow\chains^\R(\balls_n)$$ is equivalent, via a chain of quasi-isomorphisms between maps of operads, to the maps of operads
$$\HH(\balls_m; \R)\longrightarrow  \HH(\balls_n; \R).$$
This means that there is a diagram of operads, where $\calD_m, \calD_n$ are some intermediate operads, and the horizontal homomorphisms are quasi-isomorphisms
$$
\begin{array}{ccccc}
\chains^\R(\balls_m) & \stackrel{\simeq}{\longleftarrow} & \calD_m &\stackrel{\simeq}{\longrightarrow}& \HH(\balls_m; \R)\\
\downarrow & &\downarrow & & \downarrow \\
\chains^\R(\balls_n) &\stackrel{\simeq}{ \longleftarrow} & \calD_n &\stackrel{\simeq}{\longrightarrow}& \HH(\balls_n; \R)
\end{array}
$$
Recall that $m<n$. It follows that for each $i$, the map $\balls_m(i)\to \balls_n(i)$ is null-homotopic, and in particular induces the zero homomorphism on homology above degree zero. In other words, the map $\HH(\balls_m; \R)\longrightarrow \HH(\balls_n, \R)$ on the right side of the diagram factors through $\HH_0(\balls_n, \R)$. It follows that $\HH(\balls_n)$ is equivalent, as a module over $\HH(\balls_m)$, to the direct sum $\bigoplus_{i=0}^\infty \HH_i(\balls_n)$. Next, we can consider $\HH(\balls_n)$ as a module over $\calD_m$ via the pull-back functor. Then, using the identification of the category of modules with a category of diagrams, we can use the derived left Kan extension (which is the derived left adjoint to the pull-back functor) along the map of operads $\calD_m\rightarrow \chains^\R(\balls_m)$ to get a module over $\chains^\R(\balls_m)$. Let us denote this module $ {\mathrm L}\HH(\balls_n; \R)$. It follows from the formality theorem that $ {\mathrm L}\HH(\balls_n; \R)$ is weakly equivalent to $\chains^\R(\balls_n)$ as module over $\chains^\R(\balls_m)$. On the other hand, since $ \HH(\balls_n; \R)$ splits, in the category of modules over $ \HH(\balls_m; \R)$, as a direct sum of modules each concentrated in a single homological degree, it follows that $ {\mathrm L}\HH(\balls_n; \R)$ also is equivalent, as a $\chains(\balls_m)$-module, to a direct sum of modules, with each summand concentrated in a single homological degree. It follows, using Corollary~\ref{cor: formal} and the identification of the category of modules over $\chains(\balls_m)$ with a category of diagrams enriched over $\Ch$, that $ {\mathrm L}\HH(\balls_n; \R)$ is equivalent to $\HH\left( {\mathrm L}\HH(\balls_n; \R)\right)$ as a module over $\chains^\R(\balls_m)$.
 It now follows that $\chains^\R(\balls_n)$ itself 
is fact equivalent to $\HH(\balls_n; \R)$ as a module over $ \chains(\balls_m)$. 

Substituting the equivalence of $\balls_m$-modules $\chains^\R(\balls_n)\simeq \HH(\balls_n; \R)$ into the formulas of Proposition~\ref{prop: chainlevel} gives the desired result over $\R$. That is, we proved that 
$$
\ET_k \left(\chains^\R(\Ebarc(\R^m, \R^n)) \right) \simeq \underset{\chains(\balls_m)}{\hWBimod_{\le k}}(\chains(\balls_m), \HH(\balls_n;\R)).
$$
To finish the proof we observe that there are equivalences
$$
\ET_k \left(\chains^\R(\Ebarc(\R^m, \R^n)) \right) \simeq
\ET_k \left(\chains^\Q(\Ebarc(\R^m, \R^n)) \right) \otimes \R$$
and
$$
\underset{\chains(\balls_m)}{\hWBimod_{\le k}}(\chains(\balls_m), \HH(\balls_n;\R))\simeq \underset{\chains(\balls_m)}{\hWBimod_{\le k}}(\chains(\balls_m), \HH(\balls_n;\Q))\otimes \R.$$
These equivalences follow from an easy compactness argument similar to one done in~\cite{ALV}. It follows that $\ET_k \left(\chains^\Q(\Ebarc(\R^m, \R^n)) \right)$ and $\underset{\chains(\balls_m)}{\hWBimod_{\le k}}(\chains(\balls_m), \HH(\balls_n;\Q))$ are rational chain complexes that become equivalent after tensoring with $\R$. Therefore they are abstractly equivalent as rational complexes.
 \end{proof}

\section{From the little disks operad to the commutative operad}\label{section: change of operads}
Our goal in this section is to show that the formulas in Proposition~\ref{prop: first reduction} can be rewritten in terms of modules over the commutative operad instead of modules over the operad $\chains(\balls_m)$. Let us review again the formulas of the proposition. Taking $k=\infty$, they can be written as follows.

$$
\ET_\infty\chains^\Q(\Ebar(M, \R^n)) \simeq \underset{\chains(\balls_m)}{\hRmod}(\chains(\sEmb(-, M)), \HH(\balls_n;\Q))$$
and
$$
\ET_\infty \chains^\Q(\Ebarc(\R^m, \R^n))  \simeq \underset{\chains(\balls_m)}{\hWBimod}(\chains(\balls_m), \HH(\balls_n;\Q)).
$$

Recall that there is a map of operads $\chains(\balls_m)\longrightarrow \Com$, and the action of $\chains(\balls_m)$ on $\HH(\balls_n;\Q)$ pulls backs from an action of $\Com$. In other words, $\HH(\balls_n; \Q)$ is in the image of the restriction functor from $\Com$-modules to $\chains(\balls_m)$-modules. Here by ``modules'' we mean either right modules or infinitesimal bimodules. In both cases, the restriction functor has a (derived) left adjoint. Since the categories of right modules and infinitesimal bimodules are equivalent to certain diagram categories, the derived left adjoint can be described as a derived left Kan extension. Let $\ind$ denote the functor from right modules over $\chains(\balls_m)$ to right modules over $\Com$ that is derived left adjoint to the forgetful functor. Similarly, let $\widetilde\ind$ denote the analogous induction functor for infinitesimal bimodules. More generally, let $\ind_{\le k}$ be the analogous functors from $k$-truncated right modules over  $\chains(\balls_m)$ to $k$-truncated right modules over $\Com$. Finally, let $\widetilde \ind_{\le k}$ be the analogous functor between categories of $k$-truncated infinitesimal bimodules.
We have the following tautological consequence of Proposition~\ref{prop: first reduction}.
\begin{corollary}\label{cor: left induction}
Suppose that $2m+1\le n$, $M$ is an open subspace of $\R^m$, and $\R^m$ is a subspace of $\R^n$. Then there are equivalences
$$
\ET_k\chains^\Q(\Ebar(M, \R^n)) \simeq \underset{\Com}{\hRmod_{\le k}}\left(\ind_{\le k}\chains(\sEmb(-, M)), \HH(\balls_n;\Q)\right)$$
and
$$
\ET_k \chains^\Q(\Ebarc(\R^m, \R^n)) \simeq \underset{\Com}{\hWBimod_{\le k}}\left(\widetilde \ind_{\le k}\,\chains(\balls_m), \HH(\balls_n;\Q)\right).
$$
\end{corollary}
Our next task is to describe explicitly the right $\Com$-module $\ind\chains\left(\sEmb(-, M)\right)$, and the infinitesimal $\Com$-bimodule $\widetilde \ind \,\chains\left(\balls_m\right)$. Since the singular chains functor $\chains$ commutes with homotopy colimits, it commutes with derived left Kan extensions. Therefore, it is enough to compute $\ind\left(\sEmb(-, M)\right)$ and $\widetilde \ind \left(\balls_m\right)$. Recall that a right $\Com$-module in $\Top$ is the same thing as a contravariant functor on the category $\calF$ of unpointed finite sets, and a infinitesimal $\Com$-bimodule is the same thing as a contravariant functor on the category $\Gamma$ of pointed finite sets (Corollaries~\ref{cor: right com modules are F modules} and~\ref{cor: weak com bimods are gamma mods} respectively). Let $M^-$ be the right $\Com$-module corresponding to the functor from $\calF$ to $\Top$ $A\mapsto M^A$. Let $(S^m)^-$ be the infinitesimal $\Com$-bimodule corresponding to the functor from $\Gamma$ to $\Top_*$, $T\mapsto \map_*(T, S^m)$. Here $S^m$ is the one-point compactification of $\R^m$, considered as a pointed space with $\infty$ being the basepoint. In the following proposition we use the same notation for a module and its $k$-truncation.
\begin{proposition}\label{prop: explicit Kan}
For each $k\le \infty$ there is an equivalence of right $\Com$-modules.
$$\ind_{\le k}(\sEmb(-, M)) \simeq M^-,$$
and an equivalence of infinitesimal $\Com$-bimodules
$$\widetilde \ind_{\le k}\left(\balls_m\right)\simeq (S^m)^-.$$
\end{proposition}
\begin{proof}
Let us prove the first assertion first. Recall that $\calF_{\le k}({\balls_m})$ is the category whose objects are disjoint unions of at most $k$ copies of the standard $m$-ball, and whose morphism spaces are spaces of standard embeddings. $\calF_{\le k}(\Com)$ is the category of finite sets with at most $k$ elements. There is a functor $\pi_0\colon\calF_{\le k}(\balls_m)\longrightarrow\calF_{\le k}(\Com)$, which takes a disjoint union of balls to the set of path components, and a standard embedding to the map that it induces on path components.

The functor $\ind_{\le k}$ is homotopy left Kan extension along $\pi_0$. Our goal is to show that the homotopy left Kan extension along $\pi_0$ of the contravariant functor $$\sEmb(-, M)\colon \calF_{\le k}(\balls_m) \longrightarrow \Top$$ 
is equivalent to the functor
$$M^-\colon \calF_{\le k}(\Com)\longrightarrow \Top.$$

Let us consider first the case when $M$ is itself an object of $\calF_{\le k}(\balls_m)$. That is, when $M$ is the disjoint union of at most $k$ standard balls in $\R^m$. Let us relabel $M=U$ in this case. The functor $\sEmb(-, U)$ is the free functor generated by $U$. It follows, by enriched Yoneda lemma, that the strict left Kan extension of $\sEmb(-, U)$ along $\pi_0$ is the free contravariant functor generated by $\pi_0(U)$. Moreover, for free functors, homotopy left Kan extension is naturally equivalent to the strict Kan extension. Therefore, 
$$\ind\left(\sEmb(-, U)\right)\simeq\pi_0(U)^-.$$
Moreover, the path components of $U$ are contractible, so the natural map $U\longrightarrow \pi_0(U)$ is a homotopy equivalence. It follows that there is a natural equivalence, $\pi_0(U)^-\simeq U^-$. We have obtained an equivalence, natural both in $-$ and in $U$
$$\ind\left(\sEmb(-, U)\right)\simeq U^-.$$

Now let $M$ be a general open subspace of $\R^m$. We claim that if $-$ has at most $k$ components then $\sEmb(-, M)$ is equivalent to the homotopy colimit of $\sEmb(-, U)$, where $U$ ranges over the poset of subsets of $M$ that are the union of at most $k$ standard balls. To prove the claim, suppose that $-$ is the union of $i$ balls. Then $\sEmb( -, U)$ is naturally equivalent to the space $\Config{i}{U}$ of ordered $i$-tuples of distinct points in $U$. So it is enough to prove that the space $\Config{i}{M}$ is equivalent to the homotopy colimit of spaces $\Config{i}{U}$. Notice that the spaces $\Config{i}{U}$ form an open cover of $\Config{i}{M}$, and that every finite intersection of elements of this open cover is a union of elements of this open cover (in other words, the spaces $\Config{i}{U}$ form a basis for the topology on $\Config{i}{M}$). The assertion now follows from~\cite[Corollary 1.4]{DD}.
A similar argument shows that, $M^-$ is equivalent to the homotopy colimit of $U^-$, where $U$ ranges over the same category. The result for a general $M$ follows, because $\ind$ commutes with homotopy colimits.

Now let us consider the second assertion. The idea of the proof is similar. Recall that $\balls_m$ is equivalent, as an infinitesimal bimodule over $\balls_m$, to the functor $\sEmb(-, \R^m)$ (Lemma~\ref{lemma: equivalent infinitesimal bimodules} and Remark~\ref{remark: equivalent infinitesimal bimodules}). Let us recall the picture of the category $\widetilde\Gamma_{\le k}(\balls_m)$ that we are working with. An object of this category is a disjoint union of at most $k$ standard balls in $\R^m$, together with one ``antiball'', that is a copy of the complement of the unit ball in $\R^m$. So an object of $\widetilde\Gamma_{\le k}(\balls_m)$ has the form $S\boxtimes D^m$, where $S$ is a pointed finite set with at most $k$ elements besides the basepoint. Morphisms are standard embeddings between such objects. There is a functor $$\widetilde\Gamma_{\le k}(\balls_m)\longrightarrow \Gamma_{\le k}$$ that takes $S\boxtimes D^m$ to $S$. For $m>1$ this functor can be identified with $\pi_0$ - the path components functor. For $m=1$ this is not quite right, because the anti-ball has two connected components. By slight abuse of notation, we will denote this functor by $\pi_0$ in all cases. We want to show that the homotopy left Kan extension of $\sEmb(-, \R^m)$ along the functor $\pi_0\colon\widetilde\Gamma_{\le k}(\balls_m)\longrightarrow \Gamma_{\le k}$ is equivalent to the functor $\map_*(-, S^m)$.  Again, let us consider first the homotopy left Kan extension of the functor $\sEmb(-, U)$, where $U$ is an object of $\widetilde\Gamma_{\le k}({\balls_m})$. As before, this is a free functor, and so its derived left Kan extension is the functor $ \map_*(-, \pi_0(U))$. Here we can not say, as we did in the first part of the proof, that $\map_*(- , \pi_0(U))\simeq \map_*(-, U)$, because all the components of $U$ are not contractible. More specifically, the antiball is not contractible. Let us view $S^m$ as the one-point compactification of $\R^m$. For an object $U$ of $\widetilde\Gamma(\balls_m)$, let $\overline{U}$ be the space obtained from $U$ by adding the point at $\infty$ to the antiball, and consider the added point of $\overline{U}$ as the basepoint. All the components of $\overline{U}$ are contractible, and therefore $\map_*(-, \pi_0(U))\simeq \map_*(-, \overline{U})$. We obtained an equivalence
$$\widetilde\ind_{\le k}(\sEmb(-, U))\simeq \map_*(-, \overline{U}).$$
It is easy to check that the equivalence is natural in $U$. Now we can argue, exactly as in the first part of the proof, that there is a natural equivalence
$$\underset{U'}{\hocolim}\sEmb(-, U')\longrightarrow \sEmb(-, \R^m)$$ where $U'$ ranges over subspaces of $\R^m$ that are the union of at most $k$ balls and one centrally embedded antiball. The proof is similar to one given at the end of proof of Proposition~\ref{prop: explicit Kan}. There is a similar equivalence
$$\underset{U'}{\hocolim}\map_*(-, \overline{U'}) \longrightarrow \map_*( -, \overline{\R^m})=\map_*(-, S^m)$$
where $-$ is a pointed finite set with at most $k$ non-basepoint elements. Therefore we have equivalences
$$\widetilde\ind_{\le k}(\sEmb(-, \R^m)) \simeq\underset{U'}{\hocolim}  \widetilde\ind_{\le k}(\sEmb(-, U'))\simeq \underset{U'}{\hocolim}  \map_*(-, \overline{U'})\simeq \map_*(-, S^m).$$ This is what we wanted to prove.
\end{proof}

We have the following consequence of Corollary~\ref{cor: left induction} and Proposition~\ref{prop: explicit Kan}
\begin{proposition}\label{prop: getting there}
Under the same assumptions as in Corollary~\ref{cor: left induction} there are weak equivalences of rational chain complexes
$$
\ET_k\chains^\Q(\Ebar(M, \R^n)) \simeq \underset{\Com}{\hRmod_{\le k}}\left(\chains(M^-), \HH(\balls_n;\Q)\right)$$
and
$$
\ET_k \chains^\Q(\Ebarc(\R^m, \R^n)) \simeq \underset{\Com}{\hWBimod_{\le k}}\left(\chains((S^m)^-), \HH(\balls_n;\Q)\right).
$$
\end{proposition}
The first part of the proposition may be viewed as an alternative formulation of the main result of~\cite{ALV}.
As to the second part of the proposition, it can be simplified further, using Pirashvili's ``Dold-Kan'' correspondence between right $\Gamma$-modules and right $\Epi$-modules~\cite{PirashviliDold}. Recall that infinitesimal $\Com$-bimodules are the same as right $\Gamma$-modules (Corollary~\ref{cor: weak com bimods are gamma mods}). By~\cite[Theorem 3.1]{PirashviliDold}, the category of right $\Gamma$-modules (with values in an abelian category) is equivalent to the category of right $\Epi$-modules, where $\Epi$ is the category of (unpointed) finite sets and surjective maps between them. Given a right $\Gamma$-module $F$, the corresponding $\Epi$-module is defined using the cross-effect, so it may be denoted $\ce F$. For an unpointed finite set $i$, $\ce F (i)$ is defined to be the cokernel of the natural homomorphism
$$\bigoplus_{j=1}^i F((i-1)_+) \longrightarrow F(i_+).$$
Here the $i$ maps $F((i-1)_+)\longrightarrow F(i_+)$ are induced by the $i$ maps $i_+\longrightarrow (i-1)_+$ where the $j$-th map sends $j$ to the basepoint and is otherwise an order preserving isomorphism. Equivalently, $\ce F(i)$ can be defined as the total cokernel of a certain evident cubical diagram of objects of the form $F(j_+)$, where $j$ ranges over subsets of $i$. In the case when $i$ is empty, $\ce F(\emptyset)=F(*)$. It is easy to see that the $\Epi$-module $\ce\chains((S^m)^-)$ is equivalent to the $\Epi$-module $\widetilde\chains(S^{m-})$ which associates to a set $i$ the complex $\widetilde\chains(S^{mi})=\chains(S^{mi}, *)$ and where the $\Epi$-module structure is defined by the diagonal maps associated with surjective maps of sets. More generally, the $\Epi$-module $\ce\chains(X^-)$ is equivalent to $\widetilde\chains(X^{\wedge -})$.
\begin{definition}\label{def: hatH}
Define the $\Epi$-module $\hat\HH(\balls_n ; \Q)$ to be the cross-effect $\ce\HH(\balls_n; \Q)$. 
\end{definition}
Thus, the second assertion of Proposition~\ref{prop: getting there} translates, via Pirashvili's correspondence, into an equivalence
\begin{equation}\label{eq: pirashvili reduction}
\ET_k \chains^\Q(\Ebarc(\R^m, \R^n)) \simeq \underset{\Epi}{\hRmod_{\le k}}\left(\widetilde\chains(S^{m-}), \hat\HH(\balls_n;\Q)\right).
\end{equation}
In fact, things can be simplified even further, as shown in the following lemma
\begin{lemma}\label{lemma: easy formality}
The right $\Epi$-module $\widetilde\chains^\Q(S^{m-})$ is formal. That is, there is a weak equivalence of modules
$$\widetilde\chains^\Q(S^{m-})\simeq \widetilde\HH(S^{m-};\Q).$$
\end{lemma}
\begin{proof}
Let us assume that $m\ge 1$ (the case $m=0$ is trivial, but the forthcoming proof does not apply to it). Let us consider truncated modules first. Let $M$ be a right $\Epi$-module with values in chain complexes. That is, $M$ is a contravariant functor from $\Epi$ to chain complexes. Fix $k\ge 0$. Let $M_{\le k}$ be the module that agrees with $M$ on sets smaller or equal to $k$, and is $0$ on sets bigger than $k$. Clearly, this is well-defined, and there is a canonical surjective map of right $\Epi$-modules $M_{\le k}\to M_{\le k-1}$. Let $M^k$ be the kernel of the homomorphism. Clearly, $M^k$ is the right module that agrees with $M$ on the set with $k$ elements, and is zero elsewhere. There is a natural cofibration sequence
$$0\longrightarrow M^k \longrightarrow M_{\le k} \longrightarrow M_{\le k-1} \longrightarrow 0.$$
This sequence  can be viewed as a homotopy cofibration sequence of right $\Epi$-modules. It is classified by a map (in the homotopy category of right modules)
$$M_{\le k-1}\longrightarrow \Sigma M^k$$
where $\Sigma M^k$ is the point-wise shift of $M^k$.

Now let us apply this to the module $M=\widetilde\chains^\Q(S^{m-})$. We have the cofibration sequence
$$\widetilde\chains^\Q(S^{m-})^k \longrightarrow \widetilde\chains^\Q(S^{m-})_{\le k} \longrightarrow \widetilde\chains^\Q(S^{m-})_{\le k-1}\longrightarrow \Sigma \widetilde\chains^\Q(S^{m-})^k.$$
Let us calculate the derived mapping object, in the category of right modules, from
$\widetilde\chains^\Q(S^{m-})_{\le k-1}$ to $ \Sigma \widetilde\chains^\Q(S^{m-})^k.$ This is a chain complex whose $\HH_0$ gives the set of homotopy classes of maps between these objects. Recall that $ \Sigma \widetilde\chains^\Q(S^{m-})^k$ is a right $\Epi$-module concentrated in place $k$. In Section~\ref{section: Koszul spectral sequence} we will prove some general results about morphisms between right $\Epi$-modules. In particular, we will prove Corollary~\ref{cor: elementary calculation}, which implies that
$$\underset{\Omega}{\hRmod}\left(\widetilde\chains^\Q(S^{m-})_{\le k-1},  \Sigma \widetilde\chains^\Q(S^{m-})^k\right)\simeq \hom_{\Sigma_k}\left(\widetilde\chains^\Q(\Delta^kS^m),\widetilde\chains^\Q (S^{mk})\right).$$
Here $\hom_{\Sigma_k}$ stands for derived homomorphism in the category of chain complexes with an action of $\Sigma_k$ and $\Delta^kS^m$ is the fat diagonal in $S^{mk}$. By Alexander duality, this is quasi-isomorphic to $\widetilde\chains^\Q (\Sigma \Config{k}{\R^m} )^{\Sigma_k}$ (here we have used that over the rationals, invariants are equivalent to derived invariants), which is an $m-1$-connected, and therefore zero-connected complex. It follows that there are no homotopically non-trivial maps of modules from $\widetilde\chains^\Q(S^{m-})_{\le k-1}$ to $ \Sigma \widetilde\chains^\Q(S^{m-})^k$, and therefore there is a weak equivalence of right $\Epi$-modules
$$\widetilde\chains^\Q(S^{m-})_{\le k}  \simeq \widetilde\chains^\Q(S^{m-})_{\le k-1} \oplus \widetilde\chains^\Q(S^{m-})^{k} .$$
The right module $\widetilde\chains^\Q(S^{m-})^{k}$ is essentially determined by the chain complex $\widetilde\chains^\Q(S^{mk})$, together with an action of $\Sigma_k$. It is clear that there is a $\Sigma_k$-equivariant equivalence of $\Sigma_k$ complexes $\widetilde\chains^\Q(S^{mk})\simeq \widetilde\HH(S^{mk}; \Q)$ (this is in fact true integrally). Thus there is an equivalence of right $\Epi$-modules
$$\widetilde\chains^\Q(S^{m-})_{\le k}  \simeq \widetilde\chains^\Q(S^{m-})_{\le k-1} \oplus \widetilde\HH(S^{m-}; \Q)^{k},$$ and by induction we get an equivalence
$$\widetilde\chains^\Q(S^{m-})_{\le k}  \simeq \bigoplus_{i=1}^k \widetilde\HH(S^{m-}; \Q)^{i} .$$
It is easy to see that the right hand side is isomorphic to $\widetilde\HH(S^{m-}; \Q)_{\le k}$ and so we have an equivalence $$\widetilde\chains^\Q(S^{m-})_{\le k}  \simeq \widetilde\HH(S^{m-}; \Q)_{\le k}$$ for each $k$, with the restriction map $\widetilde\chains^\Q(S^{m-})_{\le k} \longrightarrow \widetilde\chains^\Q(S^{m-})_{\le k-1} $ corresponding to the restriction map  $\HH(S^{m-};\Q)_{\le k} \longrightarrow \widetilde\HH(S^{m-};\Q)_{\le k-1} $. Finally, $\widetilde\chains^\Q(S^{m-})$ may be identified with the inverse limit (which is also the homotopy inverse limit) of $\widetilde\chains^\Q(S^{m-})_{\le k}$, and similarly for the homological version, which implies that  $$\widetilde\chains^\Q(S^{m-})  \simeq \widetilde\HH(S^{m-}; \Q)$$
\end{proof}
We are finally ready to state and prove our main theorem.
\begin{theorem} \label{theorem: main theorem}
Suppose that $2m+1\le n$. Then there are weak equivalences for all $k$
$$
\ET_k \chains^\Q(\Ebarc(\R^m, \R^n)) \simeq \underset{\Epi}{\hRmod_{\le k}}\left(\widetilde\HH(S^{m-}), \hat\HH(\balls_n;\Q)\right).
$$
This includes the case $k=\infty$:
$$
\ET_\infty \chains^\Q(\Ebarc(\R^m, \R^n)) \simeq \underset{\Epi}{\hRmod}\left(\widetilde\HH(S^{m-}), \hat\HH(\balls_n;\Q)\right).
$$
It follows that if $2m+1<n$ then there is an equivalence of chain complexes
$$\chains^\Q(\Ebarmn)\simeq \underset{\Epi}{\hRmod}\left(\widetilde\HH(S^{m-}), \hat\HH(\balls_n;\Q)\right).
$$
\end{theorem}
\begin{proof}
Our starting point is~\eqref{eq: pirashvili reduction}, which says that
$$\ET_k \chains^\Q(\Ebarc(\R^m, \R^n)) \simeq\underset{\Epi}{\hRmod_{\le k}}\left(\widetilde\chains(S^{m-}), \hat\HH(\balls_n;\Q)\right).$$
Since the $\Epi$-module $\hat\HH(\balls_n;\Q)$ takes values in rational chain complexes, there is an equivalence
$$\underset{\Epi}{\hRmod_{\le k}}\left(\widetilde\chains^\Q(S^{m-}),
\hat\HH(\balls_n;\Q)\right)\simeq \underset{\Epi}{\hRmod_{\le k}}\left(\widetilde\chains(S^{m-}),
\hat\HH(\balls_n;\Q)\right).$$
By Lemma~\ref{lemma: easy formality}, $\widetilde\chains^\Q(S^{m-})$ can be replaced with $\widetilde\HH(S^{m-};\Q)$. Then again, since the target of the mapping space consists of rational chain complexes, we may replace the rational homology groups $\widetilde\HH(S^{m-};\Q)$ with the corresponding integral homology. For the last statement we require $2m+1<n$ to ensure that the underlying functor is analytic, see~\cite{WeissHomol}.
\end{proof}

\section{The splitting by homological degree}\label{section: splitting}
Our goal in the rest of the paper is to analyze the consequences of Theorem~\ref{theorem: main theorem}, and to use it to write explicit chain complexes for computing the rational homology  groups of $\Ebarc(\R^m, \R^n)$. Theorem~\ref{theorem: main theorem} expresses $\chains^\Q(\Ebarc(\R^m, \R^n))$ as a space of morphisms between right $\Epi$-modules $\widetilde\HH(S^{m-})$ and $\hat\HH(\balls_n;\Q)$. Obviously, these $\Epi$-modules split as direct sums of $\Epi$-modules concentrated in a single homological degree. For example, there is an isomorphism {\em of $\Epi$-modules}
$$\widetilde\HH(S^{m-})\cong \bigoplus_{i=0}^\infty \widetilde\HH_i(S^{m-}).$$
 We remind the reader that by $\widetilde \HH_i(X)$ we really mean the chain complex that has the $i$-th reduced homology of $X$ in degree $i$ and is zero otherwise. In fact, this decomposition is unnecessarily wasteful, because one need not consider all values of $i$, but only multiples of $m$.
 \begin{definition}
 We define, for each $m\ge 1$ and for each $s\ge 0$ the right $\Epi$-module $Q^m_s$ by the formula $Q^m_s(k)=\widetilde\HH_{ms}(S^{mk})$.
 \end{definition}
 Clearly, $Q^m_s$ is given by the following formula
$$Q^m_s(k)=\left\{\begin{array}{cc} 0 & k\ne s \\ \Z[ms] & k =s  \end{array}\right.$$
where $\Z[ms]=\widetilde\HH_{ms}(S^{ms})$ is the chain complex that has $\Z$ in dimension $ms$ and is zero otherwise. Note that $\Sigma_s$ acts trivially on $\Z[ms]$ if $m$ is even and acts by sign representation if $m$ is odd. $Q^m_s$ is in some sense a minimal non-zero right $\Epi$-module concentrated in degree $s$. It is obvious that there is an isomorphism of right $\Epi$-modules (assuming that $m\ge 1$)
$$\widetilde\HH(S^{m-})\cong\bigoplus_{s=0}^\infty Q^m_s.$$
There is a similar splitting for the right module $\hat\HH(\balls_n; \Q)$.
\begin{definition}
For each $n\ge 2$ and $t\ge 0$ define the right $\Epi$-module $\hat\HH^n_t$ by the formula
$$\hat\HH^n_t  = \hat\HH_{(n-1)t}(\balls_n; \Q).$$
\end{definition}
It is well known that $\HH(\balls_n(k))$ is concentrated in dimensions of the form $(n-1)t$. It follows that there is an isomorphism of $\Epi$-modules (assuming $n\ge 2$)

$$\hat\HH(\balls_n; \Q)\cong \bigoplus_{t=0}^\infty \hat\HH^n_t .$$
It is obvious that each one of the direct sums above is isomorphic to a direct product.
It follows that there is an isomorphism
$$\underset{\Epi}\hRmod\left(\widetilde\HH(S^{m-}), \hat\HH(\balls_n; \Q)\right)\cong\prod_{s,t=0}^\infty\underset{\Epi}\hRmod\left(Q^m_s,\hat\HH^n_t \right).$$
Thus we have the following immediate consequence of Theorem~\ref{theorem: main theorem}
\begin{corollary} \label{corollary: of main theorem}
Suppose that $2m+1\le n$. Then there are weak equivalences for all $k\le \infty$
$$
\ET_k \chains^\Q(\Ebarc(\R^m, \R^n)) \simeq \prod_{s,t=0}^\infty\underset{\Epi}\hRmod{}_{\le k}\left(Q^m_s,\hat\HH^n_t \right).
$$
It follows that
$$
\ET_\infty \chains^\Q(\Ebarc(\R^m, \R^n)) \simeq \prod_{s,t=0}^\infty\underset{\Epi}{\hRmod} \left(Q^m_s,\hat\HH^n_t \right).
$$
If $2m+1<n$ then the Taylor tower converges, and so we have an equivalence
$$
\chains^\Q(\Ebarc(\R^m, \R^n)) \simeq \prod_{s,t=0}^\infty\underset{\Epi}{\hRmod} \left(Q^m_s,\hat\HH^n_t \right).
$$
\end{corollary}

We will see in Section~\ref{section: Koszul complex} (Corollary~\ref{corollary: final homological formula}), that if $n>2m+1$ the product above can be replaced by a direct sum. 

\section{The Koszul spectral sequence for $\Epi$-modules}\label{section: Koszul spectral sequence}
Let $F, G$ be right $\Epi$-modules with values in chain complexes. Our next task is to analyze the complex of derived morphisms $$\underset{\Epi}{\hRmod}(F, G),$$ in order to understand better the consequences of Theorem~\ref{theorem: main theorem} and Corollary~\ref{corollary: of main theorem}. In this section we will review a spectral sequence for calculating the homology of $\hRmod_\Epi(F, G)$. The spectral sequence arises from filtering the category $\Epi$ by cardinality. The filtration has appeared in the literature in various guises, for example~\cite{Fresse}. We call it the Koszul spectral sequence, because the Koszul dual of $F$ makes an appearance. In the next section we will apply the general theory to the mapping complex $\hRmod(Q^m_s,\hat\HH^n_t)$. We will see that in this case the $E^1$ term of the Koszul spectral sequence is in fact a single chain complex, which is therefore quasi-isomorphic to $\hRmod(Q^m_s,\hat\HH^n_t)$, and is giving a small model for the mapping complex.

In fact, rather then analyzing the mapping complex $\underset{\Epi}{\hRmod}(F, G)$ we will focus on the dual of this complex. For this, we will assume most of the time that the right module $G$ is the dual of a left module. By a left $\Epi$-module is a {\it covariant} functor from $\Epi$ to a background category. To avoid possible confusion, we remind the reader that while a right $\Epi$-module is the same thing as a right module over the non-unital commutative operad, a left $\Epi$-module is not the same thing as a left module over the operad.

Suppose $C$ is a chain complex of $k$-modules. Let $D(C)$ be the dual chain complex $\hom_k(C, k)$. The grading of the dual complex is defined by $\hom(C,k)_n=\hom(C_{-n},k)$. In particular, if $C$ is non-negatively graded, $D(C)$ is non-positively graded. This means that we can not any longer confine ourselves to the category of non-negatively graded chain complexes. From now on, ``chain complex'' means bounded below chain complex. The model structure that we refer to is Quillen's original structure on bounded below chain complexes. Since in practice we only use chain complexes over the rationals, the model structure has a particularly simple form: fibrations (resp. cofibrations) are chain maps that are level-wise surjective (resp. injective). All objects are fibrant and cofibrant.

So let $G$ be a {\em left} $\Epi$-module with values in bounded chain complexes. Let $D(G)$ be the objectwise dual of $G$. Then $D(G)$ is a right $\Epi$-module. Let $F$ be another right $\Epi$-module. We are interested in the derived morphism object $\hRmod(F, D(G))$. Let $F \stackrel{h}{\otimes} G$ be the derived coend of the contravariant functor $F$ and the covariant functor $G$. It is well known that derived coend is a derived left adjoint to derived hom, so we have the following proposition.
\begin{proposition}\label{prop: coend and hom}
There is a natural weak equivalence
$\hRmod(F, D(G))\simeq D(F \stackrel{h}{\otimes}  G).$
\end{proposition}

For the rest of the section, we will focus on analyzing the homology of $F \stackrel{h}{\otimes}  G$. Ultimately we will work with rational chain complexes, so the homology groups of $D(F \stackrel{h}{\otimes}  G)$ are just the vector space duals of the homology of $F \stackrel{h}{\otimes}  G$.

Our approach to analyzing $F \stackrel{h}{\otimes}  G$ is to filter the category $\Epi$ by cardinality, use this to construct a filtration of $F \stackrel{h}{\otimes}  G$, and thus obtain a spectral sequence for calculating its homology groups.
\begin{definition} For each $k\ge 0$, let $\Epi_{\le k}$ be the full subcategory of $\Epi$ consisting of objects of cardinality $\le k$. Let $F$ and $G$ be a right and a left $\Epi$-module respectively. By restriction, $F$ and $G$ may also be considered as $\Epi_{\le k}$-modules. Let $F \stackrel{h}{\otimes}_{\le k}  G$ be the derived coend between $\Epi_{\le k}$-modules.
\end{definition}
One obtains a filtration of $F \stackrel{h}{\otimes}  G$ by truncated homotopy coends.
\begin{equation}\label{eq: filtration}
F \stackrel{h}{\otimes}_{\le 0}  G\longrightarrow F \stackrel{h}{\otimes}_{\le 1}  G\longrightarrow \cdots \longrightarrow F \stackrel{h}{\otimes}  G.
\end{equation}

We would like to analyze the homotopy cofiber of the map $$F \stackrel{h}{\otimes}_{\le k-1}  G\longrightarrow F \stackrel{h}{\otimes}_{\le k}  G.$$ Let $k\downarrow \Epi_{<k}$ be the category whose objects are surjective maps $k\twoheadrightarrow i$, where $i$ is a set that is strictly smaller than $k$, and whose morphisms are the evident commuting triangles. Thus, since $F$ is a contravariant functor from $\Epi$ to the category of chain complexes, it gives rise to a contravariant functor from $k\downarrow \Epi_{<k}$ to chain complexes, given on objects by the formula
$$(k\twoheadrightarrow i)\mapsto F(i).$$
Define $\Delta^kF$ to be the homotopy colimit of this functor. Note that $\Delta^kF$ has a natural action of the symmetric group $\Sigma_k$ and that there is a natural $\Sigma_k$-equivariant map $\Delta^kF\longrightarrow F(k)$.
\begin{lemma}\label{lemma: pushout}
For each $k$, there is a homotopy pushout square
$$\begin{CD}
\Delta^k F\otimes_{h\Sigma_k} G(k) @>>> F \stackrel{h}{\otimes}_{\le k-1}  G\\
@VVV @VVV \\
F(k)\otimes_{h\Sigma_k} G(k) @>>>  F \stackrel{h}{\otimes}_{\le k}  G.
\end{CD}$$
Here the symbol $\otimes_{h\Sigma_k}$ denotes the derived balanced product of two objects with an action of $\Sigma_k$. The left vertical map is induced from the map $\Delta^kF\longrightarrow F(k)$.
\end{lemma}
\begin{proof}
The lemma is elementary and essentially well-known. See, for example, the paper of Ahearn and Kuhn \cite[Lemma 3.7 and Corollary 3.8]{AhearnKuhn} for a proof of closely related result in the setting of topological spaces. The proof that we give is a variation of theirs.

Let $F\!\uparrow_{k-1}^k$ be the homotopy left Kan extension of $F$ from $\Epi_{\le k-1}$ to $\Epi_{\le k}$. It is easy to see that the restriction of $F\!\uparrow_{k-1}^k$ to $\Epi_{\le k-1}$ is equivalent to the restriction of $F$ to $\Epi_{\le k-1}$. On the other hand $F\!\!\uparrow_{k-1}^k(k)=\Delta^k F.$ Let $\underline{\Delta^k F}$ be the right $\Epi_{\le k}$-module defined by $\underline{\Delta^k F}(i)=0$ for $i<k$ and $\underline{\Delta^k F}(k)=\Delta^kF.$ Similarly, let $\underline{F(k)}$ be the $\Epi_{\le k}$-module defined by $\underline{F(k)}(i)=0$ for $i<k$ and $\underline{F(k)}(k)=F(k).$ There is a natural square of $\Epi_{\le k}$-modules
$$\begin{CD}
\underline{\Delta^kF} @>>> F\!\uparrow_{k-1}^k \\
@VVV @VVV \\
\underline{F(k)} @>>> F
\end{CD}$$
and we claim that this is a homotopy pushout diagram of $\Epi_{\le k}$-modules. Indeed, for $i<k$, evaluating this square on $i$ yields the square
$$\begin{CD}
0 @>>> F\!\uparrow_{k-1}^k(i) \\
@VVV @VVV \\
0 @>>> F(i)
\end{CD}$$
which is a homotopy pushout square, because the map $ F\!\uparrow_{k-1}^k(i)\longrightarrow F(i)$ is an equivalence. Evaluating the square at $k$ yields the square
$$\begin{CD}
\Delta^kF @>>> \Delta^k F \\
@VVV @VVV \\
F(k) @>>> F(k)
\end{CD}$$
which is also a homotopy pushout square. It follows that there is a homotopy pushout square
$$\begin{CD}
\underline{\Delta^kF} \stackrel{h}{\otimes}_{\le k}  G @>>> F\!\uparrow_{k-1}^k  \stackrel{h}{\otimes}_{\le k}  G\\
@VVV @VVV \\
\underline{F(k)} \stackrel{h}{\otimes}_{\le k}  G @>>> F \stackrel{h}{\otimes}_{\le k}  G.
\end{CD}$$
The lemma now follows from the following three easily verified equivalences
$$\underline{\Delta^kF} \stackrel{h}{\otimes}_{\le k}  G \simeq \Delta^kF \otimes_{h\Sigma_k} G(k),$$
$$\underline{F(k)} \stackrel{h}{\otimes}_{\le k}  G \simeq F(k) \otimes_{h\Sigma_k} G(k),$$
and
$$F\!\uparrow_{k-1}^k  \stackrel{h}{\otimes}_{\le k}  G \simeq F \stackrel{h}{\otimes}_{\le k-1}  G.$$
\end{proof}
\begin{definition}
Let $F$ be a right $\Epi$-module. For $k\ge 0$, let $\Koszul F(k)$ be the homotopy cofiber of the natural map $\Delta^kF\longrightarrow F(k)$. We call the sequence $\{\Koszul F(k)\}_{k=0}^\infty$ the Koszul dual of $F$.
\end{definition}
\begin{corollary}\label{cor: description of cofiber}
The homotopy cofiber of the map $$F \stackrel{h}{\otimes}_{\le k-1}  G\longrightarrow F \stackrel{h}{\otimes}_{\le k}  G$$ is equivalent to $$\Koszul F(k)\otimes_{h\Sigma_k} G(k).$$
\end{corollary}
\begin{remark}
We will point out that in the case when $F$ and $G$ are two right $\Epi$-modules, one can analyze the derived complex of morphisms $\hRmod(F, G)$ in an analogous way. Namely, one can filter the category $\Epi$ by cardinality, and obtain a tower of fibrations converging to $\hRmod(F, G)$:
$$\hRmod(F, G)\longrightarrow \cdots \longrightarrow \hRmod_{\le k}(F, G) \longrightarrow \cdots.$$
There is a description of the homotopy fiber of the restriction map $\hRmod_{\le k}(F, G) \longrightarrow \hRmod_{\le k-1}(F, G)$ that is analogous to Corollary~\ref{cor: description of cofiber}, and is proved in the same way.
\begin{lemma}\label{lem: description of fiber}
Let $F$, $G$ be right $\Epi$-modules. The homotopy fiber of the map $$\hRmod_{\le k}(F, G)\longrightarrow \hRmod_{\le k-1}(F, G)$$ is equivalent to $$\hom(\Koszul F(k), G(k))^{h\Sigma_k}.$$
\end{lemma}
\end{remark}
\begin{corollary}\label{cor: elementary calculation}
Let $n\ge 0$ be an integer. Let $F$ and $G$ be right $\Epi$-modules, and suppose that $G(i)\simeq 0$ for $i\ne n$ then
$$\underset{\Epi}{\hRmod}(F, G)\simeq \hom(\Koszul F(n), G(n))^{h\Sigma_n}.$$
\end{corollary}
\begin{proof}
It follows from the assumption that $\hom(\Koszul F(i), G(i))^{h\Sigma_i}\simeq 0$ for $i\ne n$. It follows that all the fibers in the tower of fibrations $$\cdots \longrightarrow \hRmod_{\le i}(F, G)\longrightarrow \hRmod_{\le i-1}(F, G)\longrightarrow \cdots $$ are trivial, except for the $n$-th fiber. It follows that the inverse homotopy limit of the tower, which is equivalent to $\underset{\Epi}{\hRmod}(F, G)$, is equivalent to the $n$-th fiber. But the $n$-th fiber is equivalent to
$$\hom(\Koszul F(n), G(n))^{h\Sigma_n}.$$
\end{proof}

Now we go back to assuming that $F$ is a right $\Epi$-module and $G$ is a left module.
\begin{definition}
Corollary~\ref{cor: description of cofiber} gives rise to a first quadrant spectral sequence converging to the homology groups of $F\stackrel{h}{\otimes} G$. The $k$-th column of the spectral sequence is given by the homology groups of $\Koszul F(k)\otimes_{h\Sigma_k} G(k),$ shifted up by $k$. We will call it the {\it Koszul spectral sequence}.
\end{definition}
The first differential in the Koszul spectral sequence is induced by  maps, for each $k\ge 1$
\begin{equation}\label{eq: connecting map}
\Koszul F(k)\otimes_{h\Sigma_k} G(k)\longrightarrow \Sigma\Koszul F(k-1)\otimes_{h\Sigma_{k-1}}G(k-1)
\end{equation}
which are the connecting maps associated with the filtration~\eqref{eq: filtration}. Our next task is to describe these maps explicitly. For this, we need to describe certain structure maps present in Koszul duals of right $\Epi$-modules. Consider the homotopy cofibration sequence
$$F(k)\longrightarrow \Koszul F(k) \longrightarrow \Sigma \Delta^kF.$$ Note that the maps in this sequence are $\Sigma_k$-equivariant. Recall that, by definition
$$\Delta^k F= \underset{k\twoheadrightarrow i\in k\downarrow \Epi_{<k}}{\hocolim} F(i).$$
Let $(k\downarrow\Epi_{<k-1})\subset (k\downarrow\Epi_{<k})$ be the full subcategory consisting of arrows $k\twoheadrightarrow i$ where $i< k-1$. There is a natural inclusion
$$ \underset{k\twoheadrightarrow i\in k\downarrow \Epi_{<k-1}}{\hocolim} F(i)\hookrightarrow  \underset{k\twoheadrightarrow i\in k\downarrow \Epi_{<k}}{\hocolim} F(i).$$
It is not difficult to show that the quotient of this inclusion is naturally equivalent to $\sur(k, k-1)_+ \wedge_{\Sigma_{k-1}} \Koszul F(k-1).$ Here $\sur(m,n)$ denotes the set of surjective functions from the standard set with $m$ elements to the standard set with $n$ elements. Thus we obtain a natural map
$$\Delta^k F\longrightarrow \sur(k, k-1)_+ \wedge_{\Sigma_{k-1}} \Koszul F(k-1).$$
From here, we obtain the following composite map
\begin{equation}\label{eq: quotient}
\Koszul F(k)\longrightarrow\Sigma \Delta^k F \longrightarrow \sur(k, k-1)_+ \wedge_{\Sigma_{k-1}} \Sigma \Koszul  F(k-1).
\end{equation}
\begin{remark}\label{remark: coalgebra structure}
The existence of the map~\eqref{eq: quotient} is closely related to the fact that $\Koszul F$ is the Koszul dual of a right $\Epi$-module $F$, and therefore is a right comodule over a suitable version of the Lie cooperad. See the papers of Fresse~\cite{Fresse} and Ching~\cite{Ching} for more detail.
\end{remark}

The following lemma is an exercise in manipulating colimits.
\begin{lemma}\label{lemma: connecting map}
The connecting map~\eqref{eq: connecting map} is equivalent to the following composition of maps
\begin{multline*}
\Koszul F(k)\otimes_{h\Sigma_k} G(k)\longrightarrow \left(\sur(k, k-1)_+ \wedge_{\Sigma_{k-1}} \Sigma \Koszul  F(k-1)\right) \otimes_{h\Sigma_k} G(k) \stackrel{\simeq}{\longrightarrow } \\  \stackrel{\simeq}{\longrightarrow}\Sigma\Koszul F(k-1) \otimes_{h\Sigma_{k-1}}\left(G(k)\wedge_{\Sigma_k} \sur(k, k-1)_+\right) \longrightarrow \Sigma\Koszul F(k-1)\otimes_{h\Sigma_{k-1}}G(k-1)
\end{multline*}
Here the first map is induced by the composed map~\eqref{eq: quotient}, the second map is just regrouping, and the last map is induced by the $\Sigma_{k-1}$-equivariant map
$G(k)\otimes_{h\Sigma_k} \sur(k, k-1)_+\longrightarrow G(k-1)$ which arises from the left $\Epi$-module structure on $G$.
\end{lemma}
Since $F$ and $G$ take values in the category of chain complexes, where suspension is invertible, the connecting map can also be written as having the form $$\Sigma^{-1} \Koszul F(k)\otimes_{h\Sigma_k} G(k)\longrightarrow \Koszul F(k-1)\otimes_{h\Sigma_{k-1}} G(k-1).$$
We obtain the following proposition.
\begin{proposition}\label{proposition: Koszul spectral sequence}
Filtration by cardinality gives rise to a spectral sequence (the Koszul spectral sequence), calculating $\HH(F\stackrel{h}{\otimes} G)$. The first page of the Koszul spectral sequence has the following form.
\begin{multline} \label{eq: Koszul complex}
\HH(\Koszul F(0)\otimes G(0))\longleftarrow \HH(\Sigma^{-1}\Koszul F(1)\otimes G(1))\longleftarrow \\ \longleftarrow \HH(\Sigma^{-2}\Koszul F(2)\otimes_{h\Sigma_2}G(2))\longleftarrow  \cdots\longleftarrow \HH(\Sigma^{-k}\Koszul F(k)\otimes_{h\Sigma_k} G(k))\longleftarrow \cdots
\end{multline}
Here each term describes a column in the spectral sequence page, the homomorphisms constitute the first differential, and they are induced by the map described in Lemma~\ref{lemma: connecting map}
\end{proposition}
We will sometimes like to think of the first page of the Koszul spectral sequence, together with the first differential, as a chain complex of graded abelian groups.

\section{The Koszul dual of $Q^m_s$}\label{section: Koszul dual}
Our goal for the rest of the paper is to apply the theory of the previous section to the calculation of the homology of the mapping complexes $\underset{\Epi}{\hRmod}\left(Q^m_s,\hat\HH^n_t \right)$ that appear in Corollary~\ref{corollary: of main theorem}. Thus we need to describe the module $\hat\HH^n_t$, as well as the Koszul dual of the module $Q^m_s$, as explicitly as possible.

In this section we describe the Koszul dual of $Q^m_s$, which we continue denoting $\Koszul Q^m_s$. Recall that the right $\Epi$-module $Q^m_s$ is defined by the formula
$$Q^m_s=\widetilde \HH_{ms}(S^{m-}).$$
Recall that $\Delta^kS^m$ is the fat diagonal in $S^{mk}$. Let $S^{mk}/\Delta^kS^m$ be the quotient space. It is well known that the reduced homology of this space is concentrated in degrees $k+s(m-1)$, where $s=0,1,\ldots, k$.
\begin{proposition}\label{prop: description of KQ}
For each $k$, $\Koszul Q^m_s(k)$ is quasi-isomorphic to a chain complex concentrated in a single homological degree. Moreover, if $m>1$ then there is an equivalence of chain complexes
$$\Koszul Q^m_s(k)\cong \widetilde\HH_{k+s(m-1)}\left(S^{mk}/\Delta^kS^m\right).$$
\end{proposition}
To prove the proposition, we will analyze the Koszul dual of a couple of right $\Epi$-modules of a more general type. 
\begin{definition}\label{definition: koszul duals}
Let $X$ be a pointed topological space. Let $\Sigma^\infty X^{\wedge -}$ be the right $\Epi$-module defined by the formula $i\mapsto \Sigma^\infty X^{\wedge i}$. The convention is that $X^{\wedge 0}=S^0$. Given a surjective map $i\twoheadrightarrow j$, the corresponding map $\Sigma^\infty X^{\wedge j} \to \Sigma^\infty X^{\wedge i}$ is defined using the diagonal inclusion.
\end{definition}
Let $\Delta^kX$ be the fat diagonal in $X^{\wedge k}$. It follows easily from the definition that the Koszul dual of the module $\Sigma^\infty X^{\wedge -}$ is given by the formula
$$\Koszul \Sigma^\infty X^{\wedge -}\simeq \Sigma^\infty X^{\wedge -}/\Delta^- X.$$
\begin{definition}\label{definition: I_X,s}
Fix an integer $s\ge 0$, and Let $I^X_s$ be the evident right $\Epi$-module defined by the formula
$$I^X_s(k)=\left\{\begin{array}{ll}
\Sigma^\infty X^{\wedge k} & k=s \\
* & k\ne s \end{array}\right. $$
\end{definition}
Observe that $I^X_s$ is related to $\Sigma^\infty X^{\wedge -}$ via Goodwillie differentiation. Let $F$ be a functor from the category of pointed spaces to the category of spectra (or spaces). Let $D_sF(X)$ be the $s$-th layer in the Goodwillie tower of $F$. In particular, for fixed integers $k, s$
$$D_s(\Sigma^\infty X^{\wedge k})=\left\{\begin{array}{ll}
\Sigma^\infty X^{\wedge k} & k=s \\
* & k\ne s \end{array}\right. $$
In other words, there is an equivalence of right $\Epi$-modules
$$D_s(\Sigma^\infty X^{\wedge -})\simeq I^X_s.$$
We can use this to give a description of $\Koszul I^X_s$.
\begin{lemma}\label{lemma: Koszul by differentiation}
For each $k$, there is an equivalence
$$\Koszul I^X_s(k)\simeq D_s\left(\Sigma^\infty X^{\wedge k}/\Delta^kX\right).$$
\end{lemma}
\begin{proof}
Since the Koszul dual of a right $\Epi$-module is defined using only homotopy colimits, and taking a layer in the Goodwillie tower is an operation that commutes with homotopy colimits of spectrum-valued functors, it follows that the Koszul dual of $I^X_s$ can be obtained by taking the $s$-th layer of the Koszul dual of $\Sigma^\infty X^{\wedge -}$, which is defined by the formula $\Sigma^\infty X^{\wedge -}/\Delta^- X$. In other words, we have natural equivalences
$$\Koszul I^X_s\simeq \Koszul D_s(\Sigma^\infty X^{\wedge -})\simeq  D_s(\Koszul \Sigma^\infty X^{\wedge -})\simeq  D_s\left(\Sigma^\infty X^{\wedge -}/\Delta^- X\right).$$
\end{proof}
The layers of the Goodwillie tower of the functor $\Sigma^\infty X^{\wedge k}/\Delta^k X$ are well understood. They can be described in terms of the space of partitions.
\begin{definition}
Let $\calP_k$ be the poset of partitions of the set $\{1,\ldots, k\}$ ordered by refinements. Note that it has an initial and a final object. Let $|\calP_k|$ be the geometric realization of this poset (a contractible complex). Let $\partial|\calP_k|$ be the subcomplex of $\calP_k$ that is the union of simplices that do not contain both the initial and the final object as a vertex ($\partial|\calP_k|$ can be interpreted as a kind of boundary of $|\calP_k|$). Let $T_k$ be the quotient space
$$T_k:= |\calP_k|/\partial|\calP_k|.$$
\end{definition}

The space $T_k$ often appears in the context of the calculus of functors. It is well-known that $T_k$ is homotopy equivalent to a wedge sum of $(k-1)!$ copies of $S^{k-1}$.

For any finite set $S$, we define $T_S$ to be the space constructed out of the poset of partitions of $S$ in the same way as $T_k$ is constructed out of $\{1,\ldots, k\}$. Clearly, $T_S$ is functorial with respect to isomorphisms of $S$.
\begin{definition}
Suppose $\alpha\in\sur(k,s)$ is a surjective function. Define
$$T_\alpha=T_{\alpha^{-1}(1)}\wedge \ldots \wedge T_{\alpha^{-1}(s)}.$$
\end{definition}
\begin{lemma}\label{lemma: general formula}
There is an equivalence
\begin{equation*}
\Koszul I^X_s(k)= \left(\bigvee_{\alpha\in \sur(k, s)} \Sigma^\infty \, T_{\alpha}\wedge X^{\wedge s}\right)_{\Sigma_s}.
\end{equation*}
\end{lemma}
\begin{proof}
This follows from Lemma~\ref{lemma: Koszul by differentiation} and the calculations in~\cite[Section 2]{Schwartz} (in particular, remark 2.3 in [op. cit.]).
\end{proof}
\begin{proof}[Proof of Proposition~\ref{prop: description of KQ}]
We noted already that the chain complex $ \widetilde\HH(S^{ms})=Q^m_s(s)$ is $\Sigma_s$-equivariantly quasi-isomorphic to the complex $\widetilde\chains(S^{ms})$. Since $k =s$ is the only value of  $k$ for which $Q^m_s(k)$ is not zero, it follows by an easy argument that the right $\Epi$-module $ Q^m_s$ is weakly equivalent to the right $\Epi$-module that has value  $\widetilde\chains(S^{ms})$ in degree $s$, and is zero otherwise. Using chains on spectra, and the natural equivalence $\chains(X)\simeq \chains(\Sigma^\infty X)$, we can identify $Q^m_s$ with the module that has value $\widetilde\chains(\Sigma^\infty S^{ms})$ in degree $s$, and is zero otherwise. Since the singular chains functor preserves homotopy colimits, we just need to calculate the Koszul dual of the spectrum-valued $\Epi$-module that has value $\Sigma^\infty S^{ms}$ in degree $s$ and is $*$ otherwise. This is exactly the module $I^{S^m}_s$, a special case of the module $I^{X}_s$ of Definition~\ref{definition: I_X,s}. We saw that its Koszul dual $\Koszul I^{S^m}_s$ is given by the following formula (it is a special case of Lemma~\ref{lemma: general formula})
$$\Koszul I^{S^m}_s(k)\simeq \left(\bigvee_{\alpha\in \sur(k, s)} \Sigma^\infty \, T_\alpha \wedge S^{ms}\right)_{\Sigma_s}.$$
Note that the action of $\Sigma_s$ on the set $\sur(k,s)$ is free. It follows that the module $\Koszul Q^m_s$, which is equivalent to $\widetilde\chains(\Koszul I_{S^m, s})$ is given by the following formula
\begin{equation}\label{eq: preliminary model for KQ} \Koszul Q^m_s(k)\simeq \widetilde\chains \left( \bigvee_{\alpha\in \sur(k, s)} \left(T_\alpha \wedge S^{ms}\right)\right)_{\Sigma_s}.\end{equation}
Notice that for all $\alpha\in \sur(k,s)$, $T_\alpha$ is homotopy equivalent to a wedge of spheres of dimension $k-s$. It follows that $\Koszul I^{S^m}_s(k)$ is equivalent to a wedge of spheres of dimension $k+(m-1)s$. Therefore $\Koszul Q^m_s(k)$ is homologically concentrated in dimension $k+(m-1)s$, and there is an equivalence
\begin{equation}\label{eq: final model for KQ} \Koszul Q^m_s(k)\simeq \widetilde\HH \left( \bigvee_{\alpha\in \sur(k, s)} \left(T_\alpha \wedge S^{ms}\right)\right)_{\Sigma_s}.\end{equation}

Moreover, recall that 
For each $k$, there is an equivalence (Lemma~\ref{lemma: Koszul by differentiation})
$$\Koszul I^X_s(k)\simeq D_s\left(\Sigma^\infty X^{\wedge k}/\Delta^kX\right).$$
We saw that if $X=S^m$, $D_s\left(\Sigma^\infty X^{\wedge k}/\Delta^kX\right)$ is homologically concentrated in dimension $k+(m-1)s$. Assuming $m>1$, we see that this dimension is a strictly increasing funciton of $s$. Since the spectra $D_s\left(\Sigma^\infty X^{\wedge k}/\Delta^kX\right)$ are the layers in the Taylor tower of the functor $\Sigma^\infty X^{\wedge k}/\Delta^k X$, it follows easily that when $X=S^m$ and $m>1$,  $D_s\left(\Sigma^\infty X^{\wedge k}/\Delta^kX\right)$ detects the homology of $\Sigma^\infty S^{mk}/\Delta^kS^m$ in dimension $k+(m-1)s$.
\end{proof}
Let us introduce notation that allows us to describe the complex $\Koszul Q^m_s(k)$ even more explicitly.
\begin{definition}
Let $L_k:=\widetilde \HH(T_k)$. Thus, $L_k$ is a complex concentrated in degree $k-1$ and it is isomorphic in this degree to ${\mathbb Z}^{(k-1)!}$. For $\alpha\in\sur(k,s)$, define
$$L_\alpha=\widetilde \HH(T_\alpha)=L_{\alpha^{-1}(1)}\otimes\cdots\otimes L_{\alpha^{-1}(s)}.$$
Also recall that we define $\Z[n]=\widetilde \HH(S^n)$.
\end{definition}
For any $\alpha\in\sur(k,s)$ there is a quasi-isomorphism
$$\widetilde\chains \left(T_\alpha \wedge S^{ms}\right)\simeq L_\alpha \otimes \Z[ms].$$
Using this, together with Equation~\eqref{eq: preliminary model for KQ}, we obtain the following formula
\begin{equation}\label{eq: model for KQ}
\Koszul Q^m_s (k)\simeq \left(\bigoplus_{\alpha\in \sur(k, s)} L_\alpha \otimes\Z[ms]\right)_{\Sigma_s}.
\end{equation}
\section{Homology of configuration spaces, revisited}\label{section: configurations spaces revisited}
In this section we describe more explicitly the right $\Epi$-module $\hat\HH^n_t$. The results are essentially well-known and are included for convenience. We remind the reader that  $\hat\HH^n_t=\hat\HH_{(n-1)t}(\balls_n; \Q)$, where the decoration $\hat{} $ denotes Pirashvili's cross-effect of a right $\Gamma$-module. Thus $\hat\HH^n_t(k)$ is a summand of the rational homology of the configuration space of ordered $k$-tuples of points in $\R^n$. More precisely, $\hat\HH^n_t(k)$ is the summand of the homology that is not detected in configuration spaces of fewer than $k$ points. Let $\widehat\sur(k, k-t)$ be the set of functions from $k$ to $k-t$ for which the inverse image of every point of $k-t$  has at least two elements. Elements of $\widehat\sur(k, k-t)$ correspond to partitions of $k$ with $k-t$ components which do not have singletons for components. The following formula follows easily from the results of \cite{CT:RCCS, ALV, Arone}

\begin{equation}
\hat\HH^n_t(k)\cong \left(\bigoplus_{\beta\in \widehat\sur(k, k-t)}\hom\left(L_\beta, \Q[nt]\right)\right)_{\Sigma_{k-t}}.\end{equation}
where $\Q[nt]$ denotes, as usual, the chain complex $\widetilde\HH(S^{nt};\Q)$. It is worth noting that $S^{nt}$ (and therefore also $\Q[nt]$) depends, in a sense, on the surjection $\beta$. Namely, $\beta$ induces an injective homomorphism
$\Q^{n(k-t)}\hookrightarrow \Q^{nk}$. The quotient of this homomorphism is non-canonically isomorphic to $\Q^{nt}$, and $S^{nt}$ is the one-point compactification of this $\Q^{nt}$. In particular,  this identification plays a role in the action of $\Sigma_{k-t}$ on the direct sum.

Let us use the notation $\Q[-nt]=\hom(\Q[nt], \Q)$. This is, again, a chain complex concentrated in dimension $-nt$. Note that we are dealing with actions of finite groups on rational chain complexes, so in our context derived quotients, strict quotients, derived fixed points and strict fixed points are all naturally equivalent to each other. Elementary manipulations show that the above formula for $\hat\HH^n_t(k)$ is equivalent to the following one
\begin{equation}\label{eq: model for hatHH}
\hat\HH^n_t(k)\cong \hom \left(\left(\bigoplus_{\beta\in \widehat\sur(k, k-t)}L_\beta\otimes \Q[-nt]\right)_{\Sigma_{k-t}}, \Q\right).\end{equation}

The right $\Epi$-module structure on $\hat\HH^n_t(k)$ is, in these terms, dual to a left $\Epi$-module structure on the symmetric sequence
$$k\mapsto \left(\bigoplus_{\beta\in \widehat\sur(k, k-t)} L_\beta \otimes \Q[-nt]\right)_{\Sigma_{k-t}}.$$
The left module structure was essentially described in~\cite{ALV, Arone}. We will give one description of the module structure in Section~\ref{section: forests}.
\section{The Koszul complex}\label{section: Koszul complex}
Now we are ready to describe the Koszul complex for $\underset{\Epi}{\hRmod}\left(Q^m_s, \hat\HH^n_t\right)$. (As a reminder the Koszul complex is actually dual to $\underset{\Epi}{\hRmod}\left(Q^m_s, \hat\HH^n_t\right)$, see Proposition~\ref{prop: coend and hom}.) Since right $\Epi$-module $\hat\HH^n_t$ is the dual of a left $\Epi$-module, we can use the results of Section~\ref{section: Koszul spectral sequence}, and in particular~\eqref{eq: Koszul complex}. By~\eqref{eq: Koszul complex},~\eqref{eq: model for KQ} and~\eqref{eq: model for hatHH}, the Koszul complex has the following form
$$\cdots \longleftarrow \Sigma^{-k}\left(\bigoplus_{\alpha\in \sur(k, s)} L_\alpha \otimes\Z[ms]\right)_{\Sigma_s} \otimes_{_{\Sigma_k}} \left(\bigoplus_{\beta\in \widehat\sur(k, k-t)}L_\beta\otimes \Q[-nt]\right)_{\Sigma_{k-t}}\longleftarrow \cdots $$
To shorten notation, let us introduce some abbreviations. For $\alpha\in\sur(k, s), \beta\in \widehat\sur(k, k-t)$, let $L_{\alpha, \beta}=L_{\alpha}\otimes L_{\beta}$. Also, let us abbreviate $\Sigma^{-k}L_{\alpha,\beta}\otimes \Z[ms]\otimes \Q[-nt]$ as $L_{\alpha,\beta}[ms-nt-k]$. Then the Koszul complex can be written in the following form
\begin{equation}\label{eq: dual Koszul complex}
\cdots \longleftarrow \left(\bigoplus_{\underset{\sur(k, s)\times \widehat\sur(k, k-t) } {\alpha,\beta \in} }L_{\alpha,\beta} [ms-nt-k]\right)_{\Sigma_k\times\Sigma_s\times\Sigma_{k-t}}\longleftarrow \cdots
\end{equation}
\begin{remark}\label{remark: about degree}
Recall that for $\alpha\in\sur(k,s)$, $L_\alpha$ is a chain complex concentrated in degree $k-s$. Similarly $L_\beta$ is concentrated in degree $k-(k-t)$, i.e., $t$. It follows that the $k$-th term in the complex~\eqref{eq: dual Koszul complex} is a chain complex concentrated in degree $(k-s)+t+(ms-nt-k)$, which is $(m-1)s-(n-1)t$. In particular, {\it the degree is independent of $k$}. It follows that~\eqref{eq: dual Koszul complex}, which a-priori is a chain complex of graded vector spaces, is in fact an ordinary chain complex of rational vector spaces, where all the spaces have an internal degree $(m-1)s-(n-1)t$.
\end{remark}
\begin{definition}
We let $\HH\HH^{m,n}_{s,t}$ denote the chain complex~\eqref{eq: dual Koszul complex}. We call $\HH\HH^{m,n}_{s,t}$ the Koszul complex of $\hRmod\left(Q^m_s, \hat\HH^n_t\right)$.
\end{definition}
We saw that $\HH\HH^{m,n}_{s,t}$ is a chain complex of vector spaces, which have an internal grading $(m-1)s-(n-1)t$. It follows that the dual complex $\hom\left(\HH\HH^{m,n}_{s,t},\Q\right)$ is a chain complex of vector spaces with internal grading $(n-1)t-(m-1)s$. We define the cohomology groups $\HH\HH^{m,n}_{s,t}$ to be the homology groups of the dual complex. More precisely we have isomorphisms (the last of which is not natural)
$$\HH^k(\HH\HH^{m,n}_{s,t}):=\HH_{-k}\left(\hom\left(\HH\HH^{m,n}_{s,t},\Q\right)\right)\cong\hom\left(\HH_k(\HH\HH^{m,n}_{s,t}),\Q\right)\cong \HH_k(\HH\HH^{m,n}_{s,t}),$$ where in calculation of homology we forget about the internal grading.  Combining Proposition~\ref{prop: coend and hom}, Proposition~\ref{proposition: Koszul spectral sequence} and Remark~\ref{remark: about degree} we obtain that the cohomology groups of $\HH^{m,n}_{s,t}$ are the homology groups of $\hRmod(Q^m_s, \hat\HH^n_t)$. More explicitly, we have an isomorphism for each $k\ge 0$:
\begin{equation}\label{eq: calculation of Rmod}
 \HH^k\left(\HH^{m,n}_{s,t}\right)\cong \HH_{(n-1)t-(m-1)s-k}\left(\hRmod\left(Q^m_s, \hat\HH^n_t\right)\right).\end{equation}
 \begin{remark}\label{remark: relevant degrees}
Suppose that $s, t>0$. Then $\sur(k,s)\ne \emptyset$ only if $k\ge s$. Similarly $\widehat\sur(k, k-t)\ne \emptyset$ only if $k-t>0$ and $k\ge 2(k-t)$. It follows that the complex $\HH\HH^{m,n}_{s,t}$ is non-zero only in degrees $k$ that satisfy $\max\{s, t+1\}\le k\le 2t$. It follows that $\HH_j\left(\hRmod\left(Q^m_s, \hat\HH^n_t\right)\right)$ can be non-zero only when
$$(n-3)t - (m-1)s \le j \le (n-2)t-(m-1)s-1.$$
Moreover, $\hRmod\left(Q^m_s, \hat\HH^n_t\right)$ is non-trivial only if $s\le 2t$. It follows easily that for fixed $m,n$ satisfying $m\ge 1$ and $n>2m+1$, and for each $j\ge 0$, $\HH_j\left(\hRmod\left(Q^m_s, \hat\HH^n_t\right)\right)$ is non-zero only for finitely many values of $s, t$. It follows that, assuming that $n>2m+1$, the direct product on the right hand side of the formulas in Corollary~\ref{corollary: of main theorem} are equivalent to direct sums. We obtain the following corollary
\begin{corollary}\label{corollary: final homological formula}
Suppose that $n>2m+1$. Then there are isomorphisms
\begin{equation}\label{eq:double_splitting_in_homology}
\widetilde\HH(\Ebarc(\R^m,\R^n);\Q)\cong \bigoplus_{1\le s\le 2t} \HH\left(\hRmod(Q^m_s,\hat\HH^n_t)\right)\cong  \bigoplus_{1\le s\le 2t} \HH^{(n-1)t-(m-1)s-*}(\HH\HH^{m,n}_{s,t}).
\end{equation}
\end{corollary}
\end{remark}
\begin{examples}
When $s=t=0$, $\HH\HH^{m,n}_{s,t}$ is the trivial complex which has $\Q$ in degree zero and nothing else. It detects the unreduced zero-dimensional homology of $\Ebarc(\R^m, \R^n)$. In all other cases $\HH\HH^{m,n}_{s,t}$ is non-trivial only if $s,t>0$. Furthermore it follows from Remark~\ref{remark: relevant degrees} that $\HH\HH^{m,n}_{s,t}$ is non-zero only if $s\le 2t$.

Let us consider the case $s=1$. Clearly, in this case $s\le t+1$ and, again by remark~\ref{remark: relevant degrees}, the Koszul complex  $\HH\HH^{m,n}_{1,t}$ is non-trivial only in degrees $t+1\le k\le 2t$. The non-zero portion of  $\HH\HH^{m,n}_{1,t}$ has $t$ terms, and it has the following form
$$0\leftarrow L_{t+1}\otimes L_{t+1}[m\!-\!(n\!+\!1\!)t\!-\!1]/_{\Sigma_{t+1}}\longleftarrow\cdots\longleftarrow L_{2t}\otimes L_2^{\otimes t}[m\!-\!(n\!+\!2)t]/_{\Sigma_t\wr\Sigma_{2}}\leftarrow 0 .$$
For a general $k$ the $k$-th term of $\HH\HH^{m,n}_{1,t}$ has the form
$$\left(\bigoplus_{[\beta]\in \,{\widehat\sur(k, k-t)/_{\Sigma_{k-t}}}} L_k\otimes L_\beta[m-nt-k]\right)_{\Sigma_k}.$$
This is so because $\sur(k,1)$ consists of just one point, and there is just one partition of $k$ with $1$ component. Thus the sum is indexed by irreducible partitions of $\{1,\ldots, k\}$ of excess $t$. When $k=t+1$ there is one such partition, namely the partition with one component. At the other extreme, when $k=2t$, there is again just one type of irreducible partition of excess $t$, namely the partition of type $2-2-\ldots-2$. We believe that partitions of this type are  related to ``chord diagrams'', familiar from knot theory.

In particular, when $t=1$ the complex $\HH\HH^{m,n}_{1,1}$ has only one (possibly) non-zero term, corresponding to $k=2$. The non-zero term is
$$(L_2\otimes L_2[m-n-2])_{\Sigma_2}.$$
$L_2\otimes L_2[m-n-2]$ is really just $\Q$ concentrated in dimension $m-n$. It is easy to see that the action of $\Sigma_2$ on $L_2\otimes L_2[m-n-2]$ is trivial if $n$ is even and is multiplication by $-1$ if $n$ is odd. It follows that if $n$ is odd then  $\HH\HH^{m,n}_{1,1}$ is the zero complex. If $n$ is even then $\HH\HH^{m,n}_{1,1}$ consists of a single copy of $\Q$ in degree $k=2$ and of internal dimension $m-n$. It follows that when $n$ is even $\HH\HH^{m,n}_{1,1}$ contributes a class of dimension $n-m-2$ to $\HH\left(\Ebarc(\R^m, \R^n);\Q\right)$.

Let us also consider briefly the case $s=1$, $t=2$. In this case, the complex $\HH\HH^{m,n}_{1,2}$ has two non-zero terms, corresponding to $k=3, 4$. It has the following form
$$0\leftarrow (L_3\otimes L_3[m-2n-3])_{\Sigma_3} \longleftarrow( L_4\otimes L_2\otimes L_2[m-2n-4])_{\Sigma_2\wr\Sigma_2}\leftarrow 0.$$

The complexity of $\HH\HH^{m,n}_{s,t}$ grows rapidly with $s$ and $t$.
\end{examples}

\section{A complex of forests}\label{section: forests} In this section we will show how the Koszul complex $\HH\HH^{m,n}_{s,t}$ (see~\eqref{eq: dual Koszul complex}) can be described more explicitly as a complex of forests. Recall that the direct sum over $s$ and $t$ of all these complexes computes the rational cohomology of $\Ebarc(\R^m,\R^n)$, see Corollary~\ref{corollary: final homological formula}. The internal grading of each summand $\HH\HH^{m,n}_{s,t}$ is $(m-1)s-(n-1)t$; the total grading, which is the sum of internal grading plus grading $k$,  is minus the homological grading of the space of embeddings.

The basic ingredient in the construction of $\HH\HH^{m,n}_{s,t}$ is the graded module (concentrated in degree $k-1$) $L_k=\widetilde \HH(T_k)$. It is well known that $L_k$ can be described as the free $\Z$-module spanned by trees with vertex set $\{1,\ldots, k\}$, modulo the Arnold relation.
More precisely, let $\calA_k$ be the free graded commutative algebra on $k \choose 2$ one-dimensional generators. Thus $\calA_k$ has a one-dimensional generator $u_{i,j}$ for each unordered pair $\{i,j\}$ of distinct indices between $1$ and $k$ (so $u_{i,j}=u_{j,i}$). Note that the generators anti-commute and their squares are zero. Let $I$ be the ideal of $\calA_k$ generated by the {\it Arnold relation}
$$u_{i,j}u_{j,k}+u_{j,k}u_{k,i}+u_{k,i}u_{i,j}=0$$ for all $i,j,k.$
$I$ is generated by homogeneous elements so $\calA_k/I$ is a graded algebra. It is well known that there is an isomorphism of graded algebras~\cite{Arn,Coh}
$$\calA_k/I\cong \HH^*(\mathrm{C}(k, \R^2)).$$
This has the following consequence, which is also well-known.
\begin{proposition}
$L_k$ is naturally isomorphic to the homogeneous degree $k-1$ part of $\calA_k/I$. Moreover, the degree $k-1$ part is generated by monomials of the form $u_{i_1,j_1}\cdots u_{i_{k-1},j_{k-1}}$ for which the graph with vertex set $\{1,\ldots,k\}$ and edge set $\{i_1,j_1\},\ldots,\{i_{k-1},j_{k-1}\}$ is connected and acyclic (i.e., a tree).
\end{proposition}
Indeed, one way to obtain the first part of the proposition is to use the natural isomorphisms
$$L_k\cong \HH_{k-1}(T_k)\cong\HH_{k+1}(S^{2k}/\Delta^kS^2)\cong \HH^{k-1}(\mathrm{C}(k, \R^2))$$
and the identification of $\calA_k/I$ with the cohomology of the configuration space $\mathrm{C}(k, \R^2)$. The second part of the proposition is proved, for example, in~\cite{CT:RCCS}.

Next we describe the Lie-comodule structure on the sequence $L_1, \ldots, L_k,\ldots$ in terms of this ``tree basis''. Recall that  for every surjective function $\beta\in\sur(k,k-1)$ there is a map
$T_k\longrightarrow \Sigma T_{k-1}$ associated with $\beta$. This map induces a homomorphism in homology
$$L_k\longrightarrow \Z\langle v\rangle \otimes L_{k-1}.$$ Here $v$ is a one-dimensional generator and $\Z\langle v\rangle$ is the free abelian group with generator $v$. To describe this homomorphism, we need to tell what it does on a generic element of the form $u_{i_1,j_1}\cdots u_{i_{k-1},j_{k-1}}$.
Since $\beta$ is a surjective function from the standard set with $k$ elements to the one with $k-1$ elements, there exist two elements $i, j$ such that $\beta(i)=\beta(j)$ and otherwise $\beta$ is a bijection. Roughly speaking the effect of $\beta$ on the tree basis for homology is determined by the following rule: if $\{i,j\}$ is not among the edges of a tree, then this tree is sent to zero. If $\{i,j\}$ is among the edges, then the tree is sent to another tree, obtained by contracting the edge $\{i,j\}$, and the edge $\{i,j\}$ is mapped to the suspension coordinate. More precisely, the rule is as follows: if $\beta(i)\ne\beta(j)$ then $u_{i,j}$ is sent to $u_{\beta(i),\beta(j)}$. If $\beta(i)=\beta(j)$ then $u_{i,j}$ is sent to $v$. This induces the following function on the generators of $L_k$: if the unordered pair $\{i,j\}$ is not among the pairs $\{i_1,j_1\},\ldots \{i_{k-1},j_{k-1}\}$, then the element $u_{i_1,j_1}\cdots u_{i_{k-1},j_{k-1}}$ is sent to zero. Otherwise, suppose that $\{i,j\}=\{i_l,j_l\}$. Then the element $u_{i_1,j_1}\cdots u_{i_{k-1},j_{k-1}}$ is sent to $$ u_{\beta(i_1),\beta(j_1)}\cdots u_{\beta(i_{l-1}),\beta(j_{l-1})} v u_{\beta(i_{l+1}),\beta(j_{l+1})}\cdots u_{\beta(i_{k-1}),\beta(j_{k-1})},$$ which is the same as $$(-1)^{l-1}v\otimes u_{\beta(i_1),\beta(j_1)}\cdots u_{\beta(i_{l-1}),\beta(j_{l-1})}u_{\beta(i_{l+1}),\beta(j_{l+1})}\cdots u_{\beta(i_{k-1}),\beta(j_{k-1})}.$$
One can desuspend the above homomorphism, to obtain a homomorphism
$$\Z\langle v^{-1}\rangle \otimes L_k\longrightarrow L_{k-1}$$
where $v^{-1}$ is a generator of degree $-1$. This homomorphism sends $v^{-1}\otimes u_{i_1,j_1}\cdots u_{i_{k-1},j_{k-1}}$ to $(-1)^{l-1}u_{\beta(i_1),\beta(j_1)}\cdots\hat u_{\beta(i_l),\beta(j_l)}\cdots u_{\beta(i_{k-1}),\beta(j_{k-1})}$ if $\{i,j\}=\{i_l,j_l\}$, and to zero if $\{i, j\}$ is not equal to any of the unordered pairs $\{i_1,j_1\},\ldots, \{i_k,j_k\}$.

Now let $\alpha\in\sur(k,s)$ be a surjective function. Recall that we define $L_\alpha=L_{\alpha^{-1}(1)}\otimes\cdots\otimes L_{\alpha^{-1}(s)}$. It follows that the module $L_\alpha$ is generated by {\it forests} with vertex set $\{1,\ldots, k\}$, whose connected components are labeled by $1,\ldots, s$, subject to the Arnold relation. Finally
recall that the $k$-th term of the complex $\HH\HH^{m,n}_{s,t}$ is isomorphic to the following graded rational vector space
$$\Sigma^{-k}\left(\bigoplus_{(\alpha\beta) \in \sur(k, s)\times \widehat\sur(k, k-t)} L_\alpha \otimes L_\beta \otimes \Z[ms]\otimes \Q[-nk]\otimes \Q[n(k-t)]\right)_{\Sigma_k\times \Sigma_s\times \Sigma_{k-t}} .$$
It follows that the $k$-th term of the Koszul complex $\HH\HH^{m,n}_{s,t}$ can be described as a quotient (by the Arnold relation) of a direct sum of vector spaces, indexed by equivalence classes (with respect to the action of $\Sigma_k$) of pairs of forests with vertex set $\{1,\ldots, k\}$, where the first forest has $s$ components and the second forest has $k-t$ components, all bigger than a point. Each summand is generated by equivalence classes (with respect to the action of the stabilizer group of the pair of forests) of monomials that are products of the following generators:
\begin{itemize}
\item a one-dimensional generator for each edge of the two forests,
\item $s$ generators of dimension $m$ corresponding to connected components of the first forest (and arising from $\Z[ms]$),
\item $k$ generators of dimension $-n$ corresponding to vertices and $k-t$ generators of dimension $n$ corresponding to connected components of the second forest to account for $\Q[-nk]$ and $\Q[n(k-t)]$.
\item and finally $k$ more generators of dimension $-1$, to account for the $k$-fold de-suspension. Note that the group $\Sigma_k$ does not permute these generators.
\end{itemize}
The order of generators can be interchanged as prescribed by graded commutativity.

The boundary homomorphism is defined as follows: fix a pair of forests with vertex set $\{1,\ldots, k\} $(representing an equivalence class of such pairs), where the two forests have $s$ components and $k-t$ non-singleton components respectively. The boundary homomorphism on the summand corresponding to this pair is given by a sum of $k\choose 2$ homomorphisms, corresponding to all the ways to glue together two elements of $\{1,\ldots, k\}$. The only summands that are not zero are the ones that contract an edge in the first forest (and so preserve the number $s$ of components in this forest) and glue together two components of the second forest (and so preseve the number $t$ of edges of the second forests). The contracted edge is used to cancel one of the de-suspension coordinates, as explained above.


\begin{thebibliography}{99}
\bibitem{AhearnKuhn} S. Ahearn and N. Kuhn. Product and other fine structure in polynomial resolutions of mapping spaces. {\em Algebr. Geom. Topol.}  2 (2002), 591--647.
\bibitem{Arn} V. Arnol'd. The cohomology ring of the group of colored braids. (Russian) {\it Mat. Zametki}  5,  1969, pp 227--231.
\bibitem{Schwartz} G. Arone. A note on the homology of $\Sigma_n$, the Schwartz genus, and solving polynomial equations. {\em An alpine anthology of homotopy theory}, 1--10, Contemp. Math., 399, Amer. Math. Soc., Providence, RI, 2006.
\bibitem{Arone} G. Arone, Derivatives of embedding functors. I. The stable case. {\em J. Topol.} 2 (2009), no. 3, 461--516.
\bibitem{ALV} G. Arone, P. Lambrechts, and I. Voli\`c. Calculus of functors, operad formality, and rational homology of embedding spaces. {\em Acta Math.} 199 (2007), no. 2, 153--198.
\bibitem{AT} G. Arone and V. Turchin, Graph-complexes computing the rational homotopy of high dimensional analogues of spaces of long knots. To appear in \emph{Ann. Inst. Fourier} 64 (2014).
\bibitem{DeBrito-Weiss} P. Boavida de Brito and M. Weiss, Manifold calculus and homotopy sheaves. To appear in \emph{Homology Homotopy Appl.}
\bibitem{Catt} A. Cattaneo, P. Cotta-Ramusino, and R. Longoni.
Configuration spaces and Vassiliev classes in any dimension.
\emph{Algeb. Geom. Topol.,} 2:~949--1000, 2002.
\bibitem{Ching} M. Ching. Bar constructions for topological operads and the Goodwillie derivatives of the identity. {\em Geom. Topol.} 9 (2005), 833--933.
\bibitem{Coh} F. R. Cohen. The homology of $C_{n+1}$ spaces. In {\it Lecture Notes in Mathematics}, Vol. 533, 1976.
\bibitem{CT:RCCS} F. R. Cohen and L. R. Taylor. On the representation theory associated to the cohomology of
configuration spaces, In: Algebraic Topology (Oaxtepec, 1991), {\em Contemp. Math.} 146, Amer.
Math. Soc., Providence, RI, 1993, pp. 91--109.
\bibitem{DD} D. Dugger and D. Isaksen, {Topological hypercovers and A1-realizations,}
{\em Math. Z.} 246 (2004), no. 4, 667–-689. 
\bibitem{EM} S. Eilenberg and S. Mac Lane
On the groups of H(Π,n). I., {\em Ann. of Math.} (2) 58, (1953) pp. 55–-106. 
\bibitem{Fresse} B. Fresse, Koszul duality of operads and homology of partition posets. In:  Homotopy theory: relations with algebraic geometry, group cohomology, and algebraic K-theory, {\em Contemp. Math.,} 346, Amer. Math. Soc., Providence, RI, 2004, pp. 115–-215.
\bibitem{Hirsch} M. Hirsch, Immersions of manifolds. {\em
Trans. Amer. Math. Soc.} 93 1959 242-–276.  
\bibitem{Kontsevich} M. Kontsevich, Ya. Soibelman, Deformations of algebras over operads and the Deligne conjecture. In Conf\'erence Mosh\'e Flato 1999, Vol. I (Dijon), volume 21 of Math. Phys. Stud., pages 255--307. Kluwer Acad.
Publ., Dordrecht, 2000.
\bibitem{LTV} P. Lambrechts, V. Turchin, and I. Volic. The rational homology of spaces of long knots in codimension $>2$. {\em Geom. Topol.} 14 (2010), no.~4, 2151--2187.
\bibitem{LV} P. Lambrechts and I. Volic. Formality of the little $N$-disks operad. To appear
in \emph{Mem. Amer. Math. Soc.}.
\bibitem{LodayVallette} J-L. Loday and B. Vallette, Algebraic operads. Grundlehren der Mathematischen Wissenschaften, 346. Springer, Heidelberg, 2012.
\bibitem{PirashviliDold} T. Pirashvili, Dold-Kan type theorem for $\Gamma$-groups.
{\em Math. Ann.} 318 (2000), no. 2, 277--298.
\bibitem{Sakai} S. Keiichi. Configuration space integrals for embedding spaces and the Haefliger invariant. {\em J. Knot Theory Ramification} 19 (2010), no.~12, 1597--1644.
\bibitem{Pryor} D. Pryor. Topological manifold calculus, Ph.D. Thesis, 2012, University of Virginia.
\bibitem{Sinha} D. Sinha. The topology of spaces of knots: cosimplicial models. {\em  Amer. J. Math.} 131 (2009), no. 4, 945--980.
\bibitem{Turchin10} V. Turchin,  Hodge-type decomposition in the homology of long knots. {\em J. Topol.} 3 (2010), no. 3, 487--534.
\bibitem{Turchin13} V. Turchin, Context-free manifold calculus and the Fulton-MacPherson Operad, {\em Algebr. Geom. Topol.} 13 (2013), 1243--1271
\bibitem{WeissEmb} M. Weiss, Embeddings from the point of view of immersion theory. I. {\em Geom. Topol.} 3 (1999), 67--101.
\bibitem{WeissHomol} M. Weiss. Homology of spaces of smooth embeddings, {\em Q. J. Math.} 55 (2004), no. 4, 499–504.
\end{thebibliography}
\end{document}